\documentclass[11pt]{amsart}
\usepackage{amsmath,amssymb,amsfonts}
\usepackage{tikz-cd}
\usepackage{amsbsy}
\usepackage{amscd}
\usetikzlibrary{matrix,arrows,decorations.pathmorphing}
\usepackage{graphicx}
\usepackage{float}
\usepackage{enumitem}
\usepackage{setspace}
\usepackage[margin=1.0in]{geometry}

\newtheorem{thm}{Theorem}[section]
\newtheorem{lem}[thm]{Lemma}
\newtheorem{cor}[thm]{Corollary}
\newtheorem{prop}[thm]{Proposition}

\theoremstyle{definition}
\newtheorem{example}[thm]{Example}
\theoremstyle{definition}
\newtheorem{examples}[thm]{Examples}
\theoremstyle{definition}
\newtheorem{defn}[thm]{Definition}
\theoremstyle{definition}
\newtheorem{remark}[thm]{Remark}

\newcommand{\mc}[1]{\mathcal{#1}}
\newcommand{\e}[1]{\emph{#1}}

\newcommand{\la}{\langle}
\newcommand{\ra}{\rangle}

\newcommand{\rmv}[1]{}

\newcommand{\hs}{\hskip10pt}

\newcommand{\Amod}{\mathbf{\mathcal{A}\hskip2pt mod}}
\newcommand{\modA}{\mathbf{mod\hskip2pt\mathcal{A}}}
\newcommand{\AmodA}{\mathbf{\mathcal{A}\hskip2pt mod\hskip2pt\mathcal{A}}}
\newcommand{\LG}{VN(G)}

\newcommand{\LO}{L^1(G)}
\newcommand{\LOQ}{L^1(\mathbb{G})}
\newcommand{\LOH}{L^{1}(\mathbb{H})}

\newcommand{\LOQH}{L^1(\widehat{\mathbb{G}})}
\newcommand{\LOQHs}{L^1_*(\widehat{\mathbb{G}})}
\newcommand{\LOQHP}{L^1(\widehat{\mathbb{G}}')}

\newcommand{\LT}{L^2(G)}
\newcommand{\LTQ}{L^2(\mathbb{G})}

\newcommand{\LI}{L^{\infty}(G)}
\newcommand{\LIQ}{L^{\infty}(\mathbb{G})}

\newcommand{\LIQH}{L^{\infty}(\widehat{\mathbb{G}})}
\newcommand{\LIHH}{L^{\infty}(\widehat{\mathbb{H}})}

\newcommand{\LIQHP}{L^{\infty}(\widehat{\mathbb{G}}')}

\newcommand{\BH}{\mc{B}(H)}
\newcommand{\Th}{\mc{T}(H)}
\newcommand{\LUC}{\mathrm{LUC}(\mathbb{G})}

\newcommand{\BLT}{\mc{B}(L^2(G))}
\newcommand{\BLTQ}{\mc{B}(L^2(\mathbb{G}))}

\newcommand{\KLTQ}{\mc{K}(L^2(\mathbb{G}))}
\newcommand{\TCQ}{\mc{T}(L^2(\mathbb{G}))}

\newcommand{\CBTCrr}{\mc{CB}_{\mc{T}_{\rhd}}(\mc{B}(L^2(\mathbb{G})))}

\newcommand{\Mcb}{M_{cb}A(G)}

\newcommand{\McbQl}{M_{cb}^l(L^1(\mathbb{G}))}
\newcommand{\QcbQl}{Q_{cb}^l(L^1(\mathbb{G}))}
\newcommand{\McbQHl}{M_{cb}^l(L^1(\widehat{\mathbb{G}}))}

\newcommand{\McbQr}{M_{cb}^r(L^1(\mathbb{G}))}
\newcommand{\QcbQr}{Q_{cb}^r(L^1(\mathbb{G}))}

\newcommand{\vphi}{\varphi}

\newcommand{\Lphi}{\Lambda_\varphi}

\newcommand{\lm}{\lambda}
\newcommand{\Lm}{\Lambda}
\newcommand{\Gam}{\Gamma}
\newcommand{\om}{\omega}
\newcommand{\Om}{\Omega}
\newcommand{\ten}{\otimes}
\newcommand{\oten}{\overline{\otimes}}
\newcommand{\hten}{\widehat{\otimes}}

\newcommand{\id}{\textnormal{id}}
\newcommand{\h}[1]{\widehat{#1}}
\newcommand{\wh}[1]{\widehat{#1}}

\newcommand{\sq}{\square}
\newcommand{\Irr}{\mathrm{Irr}(\mathbb{G})}
\newcommand{\ep}{\varepsilon}

\providecommand{\norm}[1]{\lVert#1\rVert}

\newcommand{\A}{\mathcal{A}}
\newcommand{\G}{\mathbb{G}}
\newcommand{\F}{\mathbb{F}}
\newcommand{\Hb}{\mathbb{H}}
\newcommand{\C}{\mathbb{C}}
\newcommand{\N}{\mathbb{N}}
\newcommand{\Z}{\mathbb{Z}}
\newcommand{\R}{\mathbb{R}}

\newcommand{\Q}{\mathbb{Q}}

\DeclareSymbolFont{lettersA}{U}{txmia}{m}{it}
\DeclareMathSymbol{\W}{\mathord}{lettersA}{151}

\begin{document}

\title[Inner amenability and approximation properties of LCQG]{Inner amenability and approximation properties of locally compact quantum groups}
\author{Jason Crann}
\email{jason.crann@carleton.ca}
\address{School of Mathematics and Statistics, Carleton University, Ottawa, ON, Canada H1S 5B6}


\begin{abstract} We introduce an appropriate notion of inner amenability for locally compact quantum groups, study its basic properties, related notions, and examples arising from the bicrossed product construction. We relate these notions to homological properties of the dual quantum group, which allow us to generalize a well-known result of Lau--Paterson \cite{LP1}, resolve a recent conjecture of Ng--Viselter \cite{NV}, and prove that, for inner amenable quantum groups $\G$, approximation properties of the dual operator algebras can be averaged to approximation properties of $\G$. Similar homological techniques are used to prove that $\ell^1(\G)$ is not relatively operator biflat for any non-unimodular discrete quantum group $\G$; a unimodular discrete quantum group $\G$ with Kirchberg's factorization property is weakly amenable if and only if $L^1_{cb}(\h{\G})$ is operator amenable, and amenability of a locally compact quantum group $\G$ implies $C_u(\h{\G})\cong \LOQH\hten_{\LOQH}C_0(\h{\G})$ completely isometrically. The latter result allows us to partially answer a conjecture of Voiculescu \cite{Voi} when $\G$ has the approximation property.
\end{abstract}

\begin{spacing}{1.0}

\maketitle

\section{Introduction}

The class of inner amenable locally compact groups has played an important role in the development of abstract harmonic analysis and its connection to operator algebras \cite{LP1,LR,St2}. In particular, it provides a sufficiently general class of locally compact groups for which approximation properties of $G$ can be characterized through approximation properties of its reduced $C^*$-algebra $C^*_\lm(G)$ and von Neumann algebra $VN(G)$. Indeed, within the class of inner amenable groups, amenability of $G$ is equivalent to nuclearity of $C^*_\lm(G)$ as well as injectivity of $VN(G)$ \cite{LP1}; exactness of $G$ is equivalent to exactness of $C^*_\lm(G)$ \cite{Claire}, and the Haagerup property of $G$ is equivalent to the Haagerup property of $VN(G)$ \cite{OOT}.

In \cite{GR}, a notion of inner amenability was introduced for an arbitrary locally compact quantum group $\G$ by means of the existence of a state $m\in\LIQ^*$ satisfying
\begin{equation}\label{e:1}\la m,x\star f\ra=\la m,f\star x\ra, \ \ \ f\in\LOQ, \  x\in\LIQ.\end{equation}
Although such a condition may be interesting in its own right, if $\G$ is co-amenable, then any cluster point in $\LIQ^*$ of a bounded approximate identity in $\LOQ$ will give rise to such a state. In particular, any locally compact group is an ``inner amenable'' locally compact quantum group in the sense of \cite{GR}.

In this paper we introduce a bona fide generalization of inner amenability to locally compact quantum groups. We study its basic properties, examples, and related notions of strong and topological inner amenability. Generalizing a well-known result of Lau--Paterson \cite{LP1}, we show that $\G$ is amenable if and only if $\LIQH$ is injective and $\G$ is inner amenable. We also derive sufficient conditions under which the bicrossed product of a matched pair of locally compact groups is inner amenable. 

We pursue homological manifestations of this concept in section 4, showing that inner amenability of $\G$ entails the relative 1-flatness of $\LOQH$. As an application of our techniques, we prove that $\G$ is co-amenable if and only if $C_0(\G)$ is nuclear and $\h{\G}$ is topologically inner amenable. This resolves a recent conjecture of Ng--Viselter \cite{NV}, and generalizes a recent result of Ng \cite{Ng} from locally compact groups to arbitrary Kac algebras, wherein $\G$ is co-amenable if and only if $C_0(\G)$ is nuclear and has a tracial state. The method of proof shows that it is \textit{inner amenability}, as opposed to discreteness, that underlies the original averaging technique of Haagerup. Indeed, we show that for strongly inner amenable quantum groups, weak amenability and the approximation property follow respectively from the w*CBAP and w*OAP of $\LIQH$. In particular, for inner amenable locally compact groups $G$, the w*CBAP of $VN(G)$ implies weak amenability of $G$ and the approximation property is equivalent to the w*OAP of $VN(G)$. Similar results are proved in the setting of topological inner amenability and approximation properties of $C_0(\h{\G})$. These techniques may be viewed as new tools for the development of harmonic analysis on quantum groups beyond the unimodular discrete setting (for which the above results greatly simplify). 

In section 5 we establish the self-duality of (non-relative) 1-biflatness, that is, $\LOQ$ is 1-biflat if and only if $\LOQH$ is 1-biflat. This shows, in particular, that $\ell^1(\G)$ is not relatively 1-biflat for any non-unimodular discrete quantum group $\G$. Section 6 is devoted to the question of operator amenability of $L^1_{cb}(\G)$, the $cb$-multiplier closure of $\LOQ$. For unimodular discrete quantum groups $\G$ with Kirchberg's factorization property, we show that $\G$ is weakly amenable if and only if $L^1_{cb}(\h{\G})$ is operator amenable. This result is new even for the class of weakly amenable residually finite discrete groups such that $C^*(G)$ is not residually finite-dimensional (see \cite{Bekka}). 

We finish in section 7 with a strengthening of \cite[Proposition 5.10]{C}, by showing that $C_u(\G)^*\cong\McbQl$ completely isometrically and weak*-weak* homeomorphically when $\h{\G}$ is amenable. As a corollary, when $\h{\G}$ has the approximation property, we prove that $\G$ is co-amenable if and only if $\h{\G}$ is amenable.

\section{Preliminaries}

\subsection{Operator Modules} Let $\mc{A}$ be a completely contractive Banach algebra. We say that an operator space $X$ is a right \e{operator $\mc{A}$-module} if it is a right Banach $\mc{A}$-module such that the module map $m_X:X\hten\mc{A}\rightarrow X$ is completely contractive, where $\hten$ denotes the operator space projective tensor product. We say that $X$ is \e{faithful} if for every non-zero $x\in X$, there is $a\in\mc{A}$ such that $x\cdot a\neq 0$, and we say that $X$ is \e{essential} if $\la X\cdot\mc{A}\ra=X$, where $\la\cdot\ra$ denotes the closed linear span. Note that our definition of faithfulness, the standard notion in operator modules, is opposite in nature to the usual definition in algebra. We denote by $\modA$ the category of right operator $\mc{A}$-modules with morphisms given by completely bounded module homomorphisms. Left operator $\mc{A}$-modules and operator $\mc{A}$-bimodules are defined similarly, and we denote the respective categories by $\Amod$ and $\AmodA$. Regarding terminology, in what follows we will often omit the term ``operator'' when discussing homological properties of operator modules as we will be working exclusively in the operator space category.


Let $\mc{A}$ be a completely contractive Banach algebra, $X\in\modA$ and $Y\in\Amod$. The \e{$\mc{A}$-module tensor product} of $X$ and $Y$ is the quotient space $X\hten_{\mc{A}}Y:=X\hten Y/N$, where
$$N=\la x\cdot a\ten y-x\ten a\cdot y\mid x\in X, \ y\in Y, \ a\in\mc{A}\ra,$$
and, again, $\la\cdot\ra$ denotes the closed linear span. It follows that (see \cite[Corollary 3.5.10]{BLM})
$$\mc{CB}_{\mc{A}}(X,Y^*)\cong N^{\perp}\cong(X\hten_{\mc{A}} Y)^*,$$
where $\mc{CB}_{\mc{A}}(X,Y^*)$ is the space of completely bounded right $\mc{A}$-module maps $\Phi:X\rightarrow Y^*$.
If $Y=\mc{A}$, then clearly $N\subseteq\mathrm{Ker}(m_X)$ where $m_X:X\hten\mc{A}\rightarrow X$ is the module map. If the induced mapping $\widetilde{m}_X:X\hten_{\mc{A}}\mc{A}\rightarrow X$ is a completely isometric isomorphism we say that $X$ is an \e{induced $\mc{A}$-module}. A similar definition applies for left modules. In particular, we say that $\mc{A}$ is \e{self-induced} if $\widetilde{m}_\mc{A}:\mc{A}\hten_{\mc{A}}\mc{A}\cong\mc{A}$ completely isometrically.

Let $\mc{A}$ be a completely contractive Banach algebra and $X\in\modA$. The identification $\mc{A}^+=\mc{A}\oplus_1\C$ turns the unitization of $\mc{A}$ into a unital completely contractive Banach algebra, and it follows that $X$ becomes a right operator $\mc{A}^+$-module via the extended action
\begin{equation*}x\cdot(a+\lm e)=x\cdot a+\lm x, \ \ \ a\in\mc{A}^+, \ \lm\in\C, \ x\in X.\end{equation*}
Let $C\geq1$. Then $X$ is \e{relatively $C$-projective} if there exists a morphism $\Phi^+:X\rightarrow X\hten\mc{A}^+$ which is a right inverse to the extended module map $m_X^+:X\hten\mc{A}^+\rightarrow X$ such that $\norm{\Phi^+}_{cb}\leq C$. When $X$ is essential, then $X$ is relatively $C$-projective if and only if there exists a morphism $\Phi:X\rightarrow X\hten\mc{A}$ satisfying $\norm{\Phi}_{cb}\leq C$ and $m_X\circ\Phi=\id_{X}$ by the operator analogue of \cite[Proposition 1.2]{DP}. We say that $X$ is \e{$C$-projective} if for every $Y,Z\in\modA$, every complete quotient morphism $\Psi:Y\twoheadrightarrow Z$, every morphism $\Phi:X\rightarrow Z$, and every $\varepsilon>0$, there exists a morphism $\widetilde{\Phi}_\varepsilon:X\rightarrow Y$ such that $\norm{\widetilde{\Phi}_\varepsilon}_{cb}< C\norm{\Phi}_{cb}+\varepsilon$  and $\Psi\circ\widetilde{\Phi}_\varepsilon=\Phi$, i.e., the following diagram commutes:
\begin{equation*}
\begin{tikzcd}
                     &Y \arrow[d, two heads, "\Psi"]\\
X \arrow[ru, dotted, "\widetilde{\Phi}_\varepsilon"] \arrow[r, "\Phi"] &Z
\end{tikzcd}
\end{equation*}

Given a completely contractive Banach algebra $\mc{A}$ and $X\in\modA$, there is a canonical completely contractive morphism $\Delta^+:X\rightarrow\mc{CB}(\mc{A}^+,X)$ given by
\begin{equation*}\Delta^+(x)(a)=x\cdot a, \ \ \ x\in X, \ a\in\mc{A}^+,\end{equation*}
where the right $\mc{A}$-module structure on $\mc{CB}(\mc{A}^+,X)$ is defined by
\begin{equation*}(\Psi\cdot a)(b)=\Psi(ab), \ \ \ a\in\mc{A}, \ \Psi\in\mc{CB}(\mc{A}^+,X), \ b\in\mc{A}^+.\end{equation*}
An analogous construction exists for objects in $\Amod$. Let $C\geq 1$. Then $X$ is \e{relatively $C$-injective} if there exists a morphism $\Phi^+:\mc{CB}(\mc{A}^+,X)\rightarrow X$ such that $\Phi^+\circ\Delta^+=\id_{X}$ and $\norm{\Phi^+}_{cb}\leq C$. When $X$ is faithful, then $X$ is relatively $C$-injective if and only if there exists a morphism $\Phi:\mc{CB}(\mc{A},X)\rightarrow X$ such that $\Phi\circ\Delta=\id_{X}$ and $\norm{\Phi}_{cb}\leq C$ by the operator analogue of \cite[Proposition 1.7]{DP}, where $\Delta(x)(a):=\Delta^+(x)(a)$ for $x\in X$ and $a\in\mc{A}$. We say that $X$ is \e{$C$-injective} if for every $Y,Z\in\modA$, every completely isometric morphism $\Psi:Y\hookrightarrow Z$, and every morphism $\Phi:Y\rightarrow X$, there exists a morphism $\widetilde{\Phi}:Z\rightarrow X$ such that $\norm{\widetilde{\Phi}}_{cb}\leq C\norm{\Phi}_{cb}$ and $\widetilde{\Phi}\circ\Psi=\Phi$, that is, the following diagram commutes:

\begin{equation*}
\begin{tikzcd}
Z \arrow[rd, dotted, "\widetilde{\Phi}"]\\
Y \arrow[u, hook, "\Psi"] \arrow[r, "\Phi"] &X
\end{tikzcd}
\end{equation*}

There is a natural categorical equivalence between $\AmodA$ and $\mathbf{mod\hskip 2pt \mc{A}^{\mathrm{op}}\hten\mc{A}}$ given by
$$axb=x\cdot(a\ten b), \ \ \ a,b\in\mc{A}, \ x\in X, \ X\in\AmodA.$$
With this identification, we obtain the following bimodule analogue of \cite[Proposition 2.3]{C}.

\begin{prop}\label{p:rel+inj} Let $\mc{A}$ be a completely contractive Banach algebra and $X\in\AmodA$. If $X$ is $C_1$-injective in $\mathbf{mod}-\C$ and is relatively $C_2$-injective in $\AmodA$, then $X$ is $C_1C_2$-injective in $\AmodA$.\end{prop}

\subsection{Locally Compact Quantum Groups} A \e{locally compact quantum group} is a quadruple $\G=(\LIQ,\Gam,\vphi,\psi)$, where $\LIQ$ is a Hopf-von Neumann algebra with co-multiplication $\Gam:\LIQ\rightarrow\LIQ\oten\LIQ$, and $\vphi$ and $\psi$ are fixed left and right Haar weights on $\LIQ$, respectively \cite{KV1,KV2}. For every locally compact quantum group $\G$, there exists a \e{left fundamental unitary operator} $W$ on $L^2(\G,\vphi)\ten L^2(\G,\vphi)$ and a \e{right fundamental unitary operator} $V$ on $L^2(\G,\psi)\ten L^2(\G,\psi)$ implementing the co-multiplication $\Gam$ via
\begin{equation*}\Gam(x)=W^*(1\ten x)W=V(x\ten 1)V^*, \ \ \ x\in\LIQ.\end{equation*}
Both unitaries satisfy the \e{pentagonal relation}; that is,
\begin{equation*}\label{penta}W_{12}W_{13}W_{23}=W_{23}W_{12}\hs\hs\mathrm{and}\hs\hs V_{12}V_{13}V_{23}=V_{23}V_{12}.\end{equation*}
By \cite[Proposition 2.11]{KV2}, we may identify $L^2(\G,\vphi)$ and $L^2(\G,\psi)$, so we will simply use $\LTQ$ for this Hilbert space throughout the paper. We denote by $R$ the unitary antipode of $\G$.

Let $\LOQ$ denote the predual of $\LIQ$. Then the pre-adjoint of $\Gam$ induces an associative completely contractive multiplication on $\LOQ$, defined by
\begin{equation*}\star:\LOQ\hten\LOQ\ni f\ten g\mapsto f\star g=\Gam_*(f\ten g)\in\LOQ.\end{equation*}
The multiplication $\star$ is a complete quotient map from $\LOQ\hten\LOQ$ onto $\LOQ$, implying
\begin{equation*}\la\LOQ\star\LOQ\ra=\LOQ.\end{equation*}
Moreover, $\LOQ$ is always self-induced. The proof follows from \cite[Theorem 2.7]{Vaes} (see \cite[Proposition 3.1]{C} for details). The canonical $\LOQ$-bimodule structure on $\LIQ$ is given by
\begin{equation*}f\star x=(\id\ten f)\Gam(x) \ \ \ \ x\star f=(f\ten\id)\Gam(x), \ \ \ x\in\LIQ, \ f\in\LOQ.\end{equation*}
A \e{left invariant mean on $\LIQ$} is a state $m\in \LIQ^*$ satisfying
\begin{equation*}\label{leftinv}\la m,x\star f \ra=\la f,1\ra\la m,x\ra, \ \ \ x\in\LIQ, \ f\in\LOQ.\end{equation*}
Right and two-sided invariant means are defined similarly. A locally compact quantum group $\G$ is said to be \e{amenable} if there exists a left invariant mean on $\LIQ$. It is known that $\G$ is amenable if and only if there exists a right (equivalently, two-sided) invariant mean (cf. \cite[Proposition 3]{DQV}). We say that $\G$ is \e{co-amenable} if $\LOQ$ has a bounded left (equivalently, right or two-sided) approximate identity (cf. \cite[Theorem 3.1]{BT}).

The \e{left regular representation} $\lm:\LOQ\rightarrow\BLTQ$ of $\G$ is defined by
\begin{equation*}\lm(f)=(f\ten\id)(W), \ \ \ f\in\LOQ,\end{equation*}
and is an injective, completely contractive homomorphism from $\LOQ$ into $\BLTQ$. Then $\LIQH:=\{\lm(f) : f\in\LOQ\}''$ is the von Neumann algebra associated with the dual quantum group $\h{\G}$. Analogously, we have the \e{right regular representation} $\rho:\LOQ\rightarrow\BLTQ$ defined by
\begin{equation*}\rho(f)=(\id\ten f)(V), \ \ \ f\in\LOQ,\end{equation*}
which is also an injective, completely contractive homomorphism from $\LOQ$ into $\BLTQ$. Then $\LIQHP:=\{\rho(f) : f\in\LOQ\}''$ is the von Neumann algebra associated to the quantum group $\h{\G}'$. It follows that $\LIQHP=\LIQH'$, and the left and right fundamental unitaries satisfy $W\in\LIQ\oten\LIQH$ and $V\in\LIQHP\oten\LIQ$ \cite[Proposition 2.15]{KV2}. Moreover, dual quantum groups always satisfy $\LIQ\cap\LIQH=\LIQ\cap\LIQHP=\C1$ \cite[Proposition 3.4]{VD}. The modular conjugations of the dual Haar weights give rise to conjugate linear isometries $J,\widehat{J}:\LTQ\rightarrow\LTQ$ satisfying
$$J\LIQ J=\LIQ'\hs\hs\mathrm{and}\hs\hs\widehat{J}\LIQH\widehat{J}=\LIQHP.$$
Moreover, the unitary $U:=\widehat{J}J$ intertwines the left and right regular representations via $\rho(f)=U\lm(f)U^*$, $f\in\LOQ$. At the level of the fundamental unitaries, this relation becomes
\begin{equation}\label{e:WV}V=\sigma(1\ten U)W(1\ten U^*)\sigma.\end{equation}
We also record the adjoint formulas $(\widehat{J}\ten J)W(\widehat{J}\ten J)=W^*$ and $(J\ten\widehat{J})V(J\ten\widehat{J})=V^*$.


If $G$ is a locally compact group, we let $\G_a=(\LI,\Gam_a,\vphi_a,\psi_a)$ denote the \e{commutative} quantum group associated with the commutative von Neumann algebra $\LI$, where the co-multiplication is given by $\Gam_a(f)(s,t)=f(st)$, and $\vphi_a$ and $\psi_a$ are integration with respect to a left and right Haar measure, respectively. The dual $\h{\G}_a$ of $\G_a$ is the \e{co-commutative} quantum group $\G_s=(\LG,\Gam_s,\vphi_s,\psi_s)$, where $\LG$ is the left group von Neumann algebra with co-multiplication $\Gam_s(\lm(t))=\lm(t)\ten\lm(t)$, and $\vphi_s=\psi_s$ is Haagerup's Plancherel weight (cf. \cite[\S VII.3]{T3}). Then $L^1(\G_a)$ is the usual group convolution algebra $\LO$, and $L^1(\G_s)$ is the Fourier algebra $A(G)$. It is known that every commutative locally compact quantum group is of the form $\G_a$ \cite[Theorem 2]{T}. Therefore, every commutative locally compact quantum group is co-amenable, and is amenable if and only if the underlying locally compact group is amenable. By duality, every co-commutative locally compact quantum group is of the form $\G_s$, which is always amenable, and is co-amenable if and only if the underlying locally compact group is amenable, by Leptin's theorem \cite{Lep}.

For a locally compact quantum group $\G$, we let $C_0(\G):=\overline{\hat{\lm}(\LOQH)}^{\norm{\cdot}}$ denote the \e{reduced quantum group $C^*$-algebra} of $\G$. We say that $\G$ is \e{compact} if $C_0(\G)$ is a unital $C^*$-algebra, in which case we denote $C_0(\G)$ by $C(\G)$. We say that $\G$ is \e{discrete} if $\LOQ$ is unital, in which case we denote $\LOQ$ by $\ell^1(\G)$. It is well-known that $\G$ is compact if and only if $\h{\G}$ is discrete, and in that case, $\ell^1(\h{\G})\cong\bigoplus_{1} \{T_{n_\alpha}(\C)\mid\alpha\in\Irr\}$,
where $T_{n_\alpha}(\C)$ is the space of $n_\alpha\times n_\alpha$ trace class operators, and $\Irr$ denotes the set of (equivalence classes of) irreducible co-representations of the compact quantum group $\G$ \cite{Wo}. For general $\G$, the operator dual $M(\G):=C_0(\G)^*$ is a completely contractive Banach algebra containing $\LOQ$ as a norm closed two-sided ideal via the map $\LOQ\ni f\mapsto f|_{C_0(\G)}\in M(\G)$. The co-multiplication satisfies $\Gam(C_0(\G))\subseteq M(C_0(\G)\ten_{\min} C_0(\G))$, where $M(C_0(\G)\ten_{\min} C_0(\G))$ is the multiplier algebra of the $C^*$-algebra $C_0(\G)\ten_{\min} C_0(\G)$.

We let $C_u(\G)$ be the \e{universal quantum group $C^*$-algebra} of $\G$, and denote the canonical surjective *-homomorphism onto $C_0(\G)$ by $\pi_{\G}:C_u(\G)\rightarrow C_0(\G)$ \cite{K}. The space $C_u(\G)^*$ then has the structure of a unital completely contractive Banach algebra
such that the map $\LOQ\rightarrow C_u(\G)^*$ given by the composition of the inclusion $\LOQ\subseteq M(\G)$ and $\pi_{\G}^*:M(\G)\rightarrow C_u(\G)^*$ is a completely isometric homomorphism, and it follows that $\LOQ$ is a norm closed two-sided ideal in $C_u(\G)^*$ \cite[Proposition 8.3]{K}.

An element $\hat{b}\in\LIQH$ is said to be a \e{completely bounded left multiplier} of $\LOQ$ if $\hat{b}\lm(f)\in\lm(\LOQ)$ for all $f\in\LOQ$ and the induced map
\begin{equation*}m_{\hat{b}}^l:\LOQ\ni f\mapsto\lm^{-1}(\hat{b}\lm(f))\in\LOQ\end{equation*}
is completely bounded on $\LOQ$. We let $\McbQl$ denote the space of all completely bounded left multipliers of $\LOQ$, which is a completely contractive Banach algebra with respect to the norm
\begin{equation*}\norm{[\hat{b}_{ij}]}_{M_n(\McbQl)}=\norm{[m^l_{\hat{b}_{ij}}]}_{cb}.\end{equation*}
Completely bounded right multipliers are defined analogously and we denote by $\McbQr$ the corresponding completely contractive Banach algebra. There is a canonical, injective completely contractive homomorphism $\tilde{\lm}:C_u(\G)^*\rightarrow\McbQl$, extending $\lm$, which maps $\mu\in C_u(\G)^*$ to the operator of left multiplication by $\mu$ on $\LOQ$. In general, given $\hat{b}\in\McbQl$, the adjoint $\Theta^l(\hat{b}):=(m_{\hat{b}}^l)^*$ defines a normal completely bounded left $\LOQ$-module map on $\LIQ$, and by  \cite[Proposition 4.2]{JNR} have the completely isometric identification
\begin{equation*}\label{e:Mcbidentifications}\Theta^l:\McbQl\cong \ _{\LOQ}\mc{CB}^{\sigma}(\LIQ).\end{equation*}
Moreover, it follows from \cite[Proposition 4.1]{JNR} that $_{\LOQ}\mc{CB}^{\sigma}(\LIQ)=_{\LOQ}\mc{CB}(C_0(\G),\LIQ)$.

It is known that $\McbQl$ is a dual operator space \cite[Theorem 3.5]{HNR2}, with predual $\QcbQl$. By the general result \cite[Lemma 1.6]{HK}, it follows from \cite[Theorem 3.3]{KR} that for arbitrary locally compact quantum groups
$$\QcbQl=\{\Om_{A,\rho}\mid A\in C_0(\G)\ten_{\min}\mc{K}(\ell^2), \ \rho\in \LOQ\hten \mc{T}(\ell^2)\},$$
where $\la\hat{b},\Om_{A,\rho}\ra=\la(\Theta^l(\hat{b})\ten\id_{K_{\infty}})(A),\rho\ra$, $\hat{b}\in\McbQl$.
We say that $\G$ is \textit{weakly amenable} if there exists an approximate identity $(\hat{f}_i)$ in $\LOQH$ which is bounded in $\McbQHl$. The infimum of bounds for such approximate identities is the Cowling--Haagerup constant of $\G$, and is denoted $\Lambda_{cb}(\G)$. We say that $\G$ has the \textit{approximation property} if there exists a net $(\hat{f}_i)$ in $\LOQH$ such that $\h{\Theta}^l(\hat{\lm}(\hat{f}_i))$ converges to $\id_{\LIQH}$ in the stable point-weak* topology. 

Let $\G$ be a locally compact quantum group. The right fundamental unitary $V$ of $\G$ induces a co-associative co-multiplication
\begin{equation*}\Gam^r:\BLTQ\ni T\mapsto V(T\ten 1)V^*\in\BLTQ\oten\BLTQ,\end{equation*}
and the restriction of $\Gam^r$ to $\LIQ$ yields the original co-multiplication $\Gam$ on $\LIQ$. The pre-adjoint of $\Gam^r$ induces an associative completely contractive multiplication on the space of trace class operators $\TCQ$, defined by
\begin{equation*}\rhd:\TCQ\hten\TCQ\ni\om\ten\tau\mapsto\om\rhd\tau=\Gam^r_*(\om\ten\tau)\in\TCQ.\end{equation*}
Analogously, the left fundamental unitary $W$ of $\G$ induces a co-associative co-multiplication
\begin{equation*}\Gam^l:\BLTQ\ni T\mapsto W^*(1\ten T)W\in\BLTQ\oten\BLTQ,\end{equation*}
and the restriction of $\Gam^l$ to $\LIQ$ is also equal to $\Gam$. The pre-adjoint of $\Gam^l$ induces another associative completely contractive multiplication 
\begin{equation*}\lhd:\TCQ\hten\TCQ\ni\om\ten\tau\mapsto\om\lhd\tau=\Gam^l_*(\om\ten\tau)\in\TCQ.\end{equation*}
The algebra $\BLTQ$ inherits a canonical $\TCQ$-bimodule structure with respect to both the left $\lhd$ and right $\rhd$ products. 

It was shown in \cite[Lemma 5.2]{HNR2} that the pre-annihilator $\LIQ_{\perp}$ of $\LIQ$ in $\TCQ$ is a norm closed two sided ideal in $(\TCQ,\rhd)$ and $(\TCQ,\lhd)$, respectively, and the complete quotient map
\begin{equation*}\label{pi}\pi:\TCQ\ni\om\mapsto f=\om|_{\LIQ}\in\LOQ\end{equation*}
is an algebra homomorphism from $\mc{T}_{\rhd}:=(\TCQ,\rhd)$, respectively, $\mc{T}_{\lhd}:=(\TCQ,\lhd)$, onto $\LOQ$. It follows that the right $\lhd$-module and left $\rhd$-module structures degenerate to an $\LOQ$-bimodule structure on $\BLTQ$:
$$ f\rhd T = (\id\ten f)\Gam^r(T), \ \ \ T\lhd f = (f\ten\id)\Gam^l(T), \ \ \ f\in\LOQ, \ T\in\BLTQ.$$

By \cite[Proposition 2.1]{KV2} the unitary antipode $R$ satisfies $R(x)=\widehat{J}x^*\widehat{J}$, for $x\in\LIQ$. It therefore extends to a *-anti-automorphism (still denoted) $R:\BLTQ\rightarrow\BLTQ$, via $R(T)=\h{J}T^*\h{J}$, $T\in\BLTQ$. The extended antipode maps $\LIQ$ and $\LIQH$ onto $\LIQ$ and $\LIQHP$, respectively, and satisfies the generalized antipode relations; that is,
\begin{equation}\label{e:R}( R\ten R)\circ\Gam^r=\Sigma\circ\Gam^l\circ R\hs\mathrm{and}
\hs( R\ten R)\circ\Gam^l=\Sigma\circ\Gam^r\circ R,\end{equation}
where $\Sigma$ is the flip map on $\BLTQ\oten\BLTQ$.

Let $\G$ and $\Hb$ be two locally compact quantum groups. Then $\Hb$ is said to be a \e{closed quantum subgroup of $\G$ in the sense of Vaes} if there exists a normal, unital, injective *-homomorphism $\gamma:\LIHH\rightarrow\LIQH$ satisfying $(\gamma\ten\gamma)\circ\Gam_{\h{\Hb}}=\Gam_{\h{\G}}\circ\gamma$. This is not the original definition of Vaes \cite[Definition 2.5]{V2}, but was shown to be equivalent in \cite[Theorem 3.3]{DKSS}.

\section{Inner Amenability}

Given a locally compact quantum group $\G$ the composition $W\sigma V\sigma\in\LIQ\oten\BLTQ$ defines a unitary co-representation of $\G$ called the \textit{conjugation co-representation}, where $\sigma$ is the flip map on $\LTQ\ten\LTQ$. When $\G=\G_a$ for some locally compact group $G$, it follows that
$$W\sigma V\sigma\xi(s,t)=\xi(s,s^{-1}ts)\Delta(s)^{1/2}, \ \ \ \xi\in L^2(G\times G), \ s,t\in G.$$ 
Thus, $W\sigma V\sigma$ is the unitary generator of the conjugation representation $\beta_2:G\rightarrow\BLT$, where
$$\beta_2(s)\xi(t)=\lm(s)\rho(s)\xi(t)=\xi(s^{-1}ts)\Delta(s)^{1/2}.$$
Following Paterson \cite[2.35.H]{Pat2}, we say that a locally compact group $G$ is \textit{inner amenable} if there exists a state $m\in\LI^*$ satisfying
\begin{equation}\label{e:PatIA}\la m,\beta_{\infty}(s)f\ra=\la m,f\ra \ \ \ s\in G, \ f\in\LI,\end{equation}
where $\beta_\infty(s)f(t)=f(s^{-1}ts)$, $s,t\in G$, $f\in\LI$, is the conjugation action on $\LI$. 

\begin{remark} In \cite{Effros}, Effros defined a discrete group $G$ to be ``inner amenable'' if there exists a conjugation invariant mean $m\in\ell^{\infty}(G)^*$ such that $m\neq\delta_e$. In what follows, inner amenability will always refer to the definition (\ref{e:PatIA}) given above. In particular, every discrete group is inner amenable.
\end{remark}

\begin{defn}\label{d:WIA} Let $\G$ be a locally compact quantum group. We say that
\begin{enumerate}[label=(\roman*)]
\item $\G$ is \e{strongly inner amenable} (see \cite{OOT}) if there exists a net $(\xi_i)$ of unit vectors such that
$$\norm{W\sigma V\sigma(\eta\ten\xi_i)-\eta\ten\xi_i}\rightarrow0, \ \ \ \eta\in\LTQ.$$
\item $\G$ is \e{inner amenable} if there exists a state $m\in\LIQH^*$ satisfying
$$\la m,\hat{x}\lhd f\ra=\la f,1\ra\la m,\hat{x}\ra, \ \ \ f\in\LOQ, \ \hat{x}\in\LIQH.$$
Such a state is said to be \e{inner invariant}.
\item $\G$ is \e{topologically inner amenable} if there exists a state $m\in C_0(\h{\G})^*$ such that
$$\la m,\hat{x}\lhd f\ra=\la f,1\ra\la m,\hat{x}\ra, \ \ \ f\in\LOQ, \ \hat{x}\in C_0(\h{\G}).$$
Such a state is said to be \e{topologically inner invariant}.
\end{enumerate}
\end{defn}

\begin{examples} The following known examples are worth mentioning. 
\begin{enumerate}[label=(\roman*)]
\item Any co-commutative quantum group $\G_s$ is trivially strongly inner amenable, as $V_s=W_a$ so that $W_s\sigma V_s\sigma=W_sW_s^*=1$. 
\item Any unimodular discrete quantum group is strongly inner amenable; the unit vector $\xi:=\Lambda_\vphi(1)$ being conjugation invariant, where $\vphi$ is the Haar trace on the compact dual.
\item It was shown in \cite{NV} that if $\G$ has trivial scaling group, and $C_0(\h{\G})$ possesses a tracial state, then $\G$ is topologically inner amenable.
\end{enumerate}
Further examples, including the bicrossed product construction will be studied below.
\end{examples}

The following proposition is standard. We include the proof for completeness.

\begin{prop}\label{p:implication} Let $\G$ be a locally compact quantum group. If $\G$ is strongly inner amenability  then it is inner amenable.
\end{prop}

\begin{proof} Let $(\xi_i)$ be a net of unit vectors asymptotically invariant under the conjugation co-representation $W\sigma V\sigma$. Passing to a subnet, we may assume that $(\om_{J\xi_i}|_{\LIQH})$ converges weak* to a state $m\in\LIQH^*$. For any $f=\om_{\xi,\eta}\in\LOQ$ and $\hat{x}\in\LIQH$, using strong inner amenability together with the adjoint relations 
$$(\widehat{J}\ten J)W(\widehat{J}\ten J)=W^*, \ \ \ (\widehat{J}\ten J)\sigma V\sigma (\widehat{J}\ten J)=\sigma V^*\sigma,$$
we have
\begin{align*}\la m,\hat{x}\lhd f\ra&=\lim_i\la \om_{J\xi_i},\hat{x}\lhd f\ra\\
&=\lim_i\la W^*(1\ten \hat{x})W(\xi\ten J\xi_i),\eta\ten J\xi_i\ra\\
&=\lim_i\la(1\ten \hat{x})W(\xi\ten J\xi_i),W(\eta\ten J\xi_i)\ra\\
&=\lim_i\la(1\ten \hat{x})(\widehat{J}\ten J)W^*(\widehat{J}\xi\ten\xi_i),(\widehat{J}\ten J)W^*(\widehat{J}\eta\ten\xi_i)\ra\\
&=\lim_i\la(1\ten \hat{x})(\widehat{J}\ten J)\sigma V\sigma(\widehat{J}\xi\ten\xi_i),(\widehat{J}\ten J)\sigma V\sigma(\widehat{J}\eta\ten\xi_i)\ra\\
&=\lim_i\la(1\ten \hat{x})\sigma V^*\sigma(\xi\ten J\xi_i),\sigma V^*\sigma(\eta\ten J\xi_i)\ra\\
&=\la m,\hat{x}\ra\la f,1\ra.\end{align*}
\end{proof}

In the commutative case, the converse of Proposition \ref{p:implication} holds.

\begin{prop}\label{p:equal} A commutative quantum group $\G_a$ is strongly inner amenable if and only if it is inner amenable if and only if its underlying group $G$ is inner amenable.\end{prop}

\begin{proof} If $\G_a$ is strongly inner amenable then it is inner amenable.  Let $m\in VN(G)^*$ satisfy $\la m,x\ra=\la m,x \lhd f\ra$ for all $x\in VN(G)$ and all states $f\in\LO$. Then for $t\in G$
$$\lm(t)x\lm(t)^* \lhd f=\int_G \lm(s^{-1}t)x\lm(t^{-1}s) f(s)ds=\int_G\lm(r)^*x\lm(r)f(tr)dr=x \lhd f_t,$$
so that
$$\la m,\lm(t)x\lm(t)^*\ra=\la m,\lm(t)x\lm(t)^* \lhd f\ra=\la m,x \lhd f_t\ra=\la m,x\ra$$
for all $x\in\LG$, $t\in G$ and states $f\in\LO$. Thus, $G$ is inner amenable by \cite[Proposition 3.2]{CT}.

Conversely, if $G$ is inner amenable then there exists a net of unit vectors $(\xi_i)$ in $\LT$ satisfying
\begin{equation*}\label{3}\norm{\lm(s)\xi_i-\rho(s)^*\xi_i}_{\LT}\rightarrow0\end{equation*}
uniformly on compact subsets of $G$. For any $\eta\in C_c(G)$ we then have
$$\la W\sigma V\sigma (\eta\ten \xi_i),\eta\ten\xi_i\ra=\iint\xi_i(ts)\overline{\xi_i}(st)\Delta(s)^{1/2}|\eta(s)|^2 \ ds \ dt\rightarrow 1.$$
It follows that $(\xi_i)$ is strongly inner invariant. 
\end{proof}

Definition \ref{d:WIA} (ii) is therefore a bona fide generalization of inner amenability to quantum groups, contrary to the definition proposed in \cite{GR}. Curiously, if one ``lifts'' the relation (\ref{e:1}) proposed in \cite{GR} to the level of $\BLT$, via
$$\la m,f \rhd T\ra=\la m, T \lhd f\ra, \ \ \ f\in\LO, \ T\in\BLT,$$
then a similar argument as in Proposition \ref{p:equal} shows that one obtains a proper generalization of inner amenability. For details, see \cite[\S5]{Cthesis}.

\begin{prop}\label{p:TIA} A commutative quantum group $\G_a$ is topologically inner amenable if and only if $C^*_\lm(G)$ possesses a tracial state if and only if the amenable radical of $G$ is open. \end{prop}

\begin{proof} The existence of a tracial state on $C_\lm^*(G)$ was recently investigated by Forrest--Spronk--Wiersma \cite{FSW}, and Kennedy--Raum \cite{KenR}, where it was shown to be equivalent to the openness of the amenable radical of $G$, that is, the largest amenable normal subgroup. 

By \cite[Proposition 3.3]{NV} $\G_a$ is topologically inner amenable if $C_\lm^*(G)$ has a tracial state. Conversely, if $m$ is an inner invariant state on $C_\lm^*(G)$, it follows as in the proof of Proposition \ref{p:equal} that $m$ is $G$-invariant under the canonical conjugation action. Viewing $m$ as a positive definite function in $B_\lm(G)$, it follows that $m(s)=m(tst^{-1})$, $s,t\in G$. A simple integral calculation then shows that $m$ is a tracial state on $C_\lm^*(G)$.\end{proof}

\begin{remark}\label{r:6} At first glance one might think that inner amenability implies topological inner amenability via the restriction of an inner invariant state $m$ on $\LIQH$ to $C_0(\h{\G})$. This is not the case however. In \cite{Su} Suzuki provided elementary constructions of non-discrete $C^*$-simple groups of the form $G=\bigoplus_{n\in\N} \Gamma_n\rtimes \prod_{n\in\N} F_n$ where $F_n$ is a finite group acting on the discrete group $\Gamma_n$ for which $C^*_\lm(\Gamma_n\rtimes F_n)$ admits a unique tracial state. In \cite[Remark 2.5, Remark 2.6 (i)]{FSW}, it was shown that each compactly generated open subgroup $H$ of $G$ is IN (meaning $H$ has a compact conjugation invariant neighbourhood of the identity), and that one may build a resulting net of normalized characteristic functions whose vector functionals cluster to an inner invariant state on $VN(G)$. Thus, $G$ is inner amenable. However, in \cite[Remark 2.5]{FSW}, the authors also show that the continuous function $m$ in $B_\lm(G)$ associated to a tracial state on $C^*_\lm(G)$ must be discontinuous. Whence, there is no tracial state on $C^*_\lm(G)$ and $G$ is not topologically inner amenable by Proposition \ref{p:TIA}. Therefore, there are ``continuity restrictions'' which forbid the restriction of an inner invariant state on $VN(G)$ to a tracial state on $C^*_\lm(G)$.

In the other direction, in general it is not possible to lift a tracial state on $C^*_\lm(G)$ to an inner invariant mean on $VN(G)$. In \cite[Remark 2.6 (ii)]{FSW} it is shown that $\R^2\rtimes \F_6$ is topologically inner amenable but not inner amenable, where $\F_6$ is viewed as a closed subgroup of $SL(2,\R)$.\end{remark}

\begin{prop}\label{p:amen=>IA} Let $\G$ be a locally compact quantum group. If $\h{\G}$ is co-amenable then $\G$ is strongly inner amenable. If $\G$ is amenable then it is inner amenable.
\end{prop}

\begin{proof} If $\h{\G}$ is co-amenable, let $E\in\LIQH^*$ be a co-unit and approximate $E$ in the weak* topology by vector states $\om_{\xi_i}$ with $\h{J}\xi_i=\xi_i$ (since $\LIQH\curvearrowright\LTQ$ is standard). Then $1=(\id\ten E)(W)=\lim_i(\id\ten\om_{\xi_i})(W)$ strongly, from which it follows that
$$\norm{W(\eta\ten\xi_i)-(\eta\ten\xi_i)}\rightarrow0, \ \ \ \eta\in\LTQ.$$
Since $\sigma V\sigma=(\h{J}\ten\h{J})W^*(\h{J}\ten\h{J})$, and $\h{J}\xi_i=\xi_i$, we also have
$$\norm{\sigma V\sigma (\eta\ten \xi_i)-(\eta\ten\xi_i)}\rightarrow0, \ \ \ \eta\in\LTQ.$$
Whence, $\G$ is strongly inner amenable.

Now, suppose $\G$ is amenable and let $m\in\LIQ^*$ be a two-sided invariant mean. We will show a stronger statement by providing a state $M\in\BLTQ^*$ such that
$$\la M,\rho\rhd T\ra=\la M,T\lhd\rho\ra, \ \ \ T\in\BLTQ, \ \rho\in\TCQ,$$
upon which restriction to $\LIQH$ is the desired state.
Letting $m$ also denote its restriction to $\LUC:=\la\LIQ\star\LOQ\ra$, let $m_0:=\rho_0\circ\Theta^r(m)\in\BLTQ^*$, where $\rho_0\in\TCQ$ is a fixed normal state, and 
$$\Theta^r:\LUC^*\ni n\mapsto (T\mapsto (\id\ten n)V^*(T\ten 1)V)\in\CBTCrr$$
is the canonical completely contractive homomorphism (see \cite[\S3]{CN} or \cite[Proposition 6.5]{HNR2}). Since $\LUC=\la\BLTQ\rhd\TCQ\ra$ \cite[Proposition 5.3]{HNR2}, one may view the map $\Theta^r$ as follows: 
$$\la\Theta^r(n)(T),\rho\ra=\la n,T\rhd\rho\ra, \ \ \ n\in\LUC^*, \ T\in\BLTQ, \ \rho\in\TCQ.$$
Consider the state $R^*(m_0)\sq m_0\in\BLTQ^*$ where $R$ is the extended unitary antipode and $\sq$ is the left Arens product on $\BLTQ^*$ extending the multiplication in $\mc{T}_\rhd=(\TCQ,\rhd)$.

Fix $\rho,\om\in\TCQ$ and $T\in\BLTQ$. Firstly, since $m\in\LUC^*$ is an invariant mean $m\sq\pi(\rho)=\la\pi(\rho),1\ra m=\la\rho,1\ra m$. Thus,
\begin{align*}\la m_0\sq\rho,T\ra&=\la m_0,\rho\rhd T\ra=\la\rho_0,\Theta^r(m)\circ\Theta^r(\pi(\rho))(T)\ra\\
&=\la\rho_0,\Theta^r(m\sq\pi(\rho))(T)\ra=\la\rho,1\ra\la\rho_0,\Theta^r(m)(T)\ra\\
&=\la\rho,1\ra\la m_0,T\ra.
\end{align*}
Hence, $m_0\sq\rho=\la\rho,1\ra m_0$. Secondly, $\Theta^r(m)$ is a conditional expectation onto $\LIQH$ commuting with both the right $\mc{T}_\rhd$- and right $\mc{T}_\lhd$-actions \cite[Theorem 4.9]{CN}. Since $\hat{x}\rhd\om=\la\om,\hat{x}\ra1$, $\hat{x}\in\LIQH$ and $\om\in\TCQ$, we also have
\begin{align*}\la m_0\sq(T\lhd\rho),\om\ra&=\la m_0,(T\lhd\rho)\rhd\om\ra=\la\rho_0,\Theta^r(m)((T\lhd\rho)\rhd\om)\ra\\
&=\la\rho_0,\Theta^r(m)(T\lhd\rho)\rhd\om\ra=\la\rho_0,1\ra\la\Theta^r(m)(T\lhd\rho),\om\ra\\
&=\la\rho_0,1\ra\la\Theta^r(m)(T)\lhd\rho,\om\ra=\la\rho_0,1\ra\la\Theta^r(m)(T),\rho\lhd\om\ra\\
&=\la\rho_0,\Theta^r(m)(T)\rhd(\rho\lhd\om)\ra=\la\rho_0,\Theta^r(m)(T\rhd(\rho\lhd\om))\ra\\
&=\la m_0,T\rhd(\rho\lhd\om)\ra=\la m_0\sq T,\rho\lhd\om\ra\\
&=\la (m_0\sq T)\lhd\rho,\om\ra.\\
\end{align*}
Thus, $m_0\sq(T\lhd\rho)=(m_0\sq T)\lhd\rho$. Putting things together, on the one hand we obtain
$$\la R^*(m_0)\sq m_0,\rho\rhd T\ra=\la R^*(m_0),(m_0\sq\rho)\sq T\ra=\la\rho,1\ra\la R^*(m_0),m_0\sq T\ra,$$
and on the other,
\begin{align*}\la R^*(m_0)\sq m_0,T\lhd\rho\ra&=\la R^*(m_0),m_0\sq(T\lhd\rho)\ra=\la R^*(m_0),(m_0\sq T)\lhd\rho\ra\\
&=\la m_0,R((m_0\sq T)\lhd\rho)\ra=\la m_0,R_*(\rho)\rhd R(m_0\sq T)\ra\\
&=\la m_0\sq R_*(\rho),R(m_0\sq T)\ra=\la R_*(\rho),1\ra\la m_0, R(m_0\sq T)\ra\\
&=\la\rho,1\ra\la R^*(m_0),m_0\sq T\ra=\la\rho,1\ra\la R^*(m_0)\sq m_0, T\ra.
\end{align*}
Therefore, $M:=R^*(m_0)\sq m_0$ is the required state.
\end{proof}

Combining \cite[Corollary 4.10]{CN} with (the proof of) Proposition \ref{p:amen=>IA}, we obtain a quantum group analogue of a well-known result of Lau--Paterson \cite[Corollary 3.2]{LP1}.

\begin{cor}\label{c:LP} A locally compact quantum group $\G$ is amenable if and only if it is inner amenable and $\LIQH$ is 1-injective in $\C\hskip2pt\mathbf{mod}$.
\end{cor}

Generalizing the hereditary property in the group setting, we show that inner amenability passes to closed quantum subgroups.

\begin{prop}\label{p:qs} Let $\G$ and $\Hb$ be locally compact quantum groups such that $\Hb$ is a closed quantum subgroup of $\G$ in the sense of Vaes. If $\G$ is inner amenable then $\Hb$ is inner amenable.\end{prop}

\begin{proof} Let $\h{m}$ be an inner invariant state in the sense of Definition \ref{d:WIA} (ii), let $\h{n}:=\h{m}\circ\gamma\in\LIHH^*$, and let $W_{\G}$ and $W_{\Hb}$ denote the left fundamental unitaries of $\G$ and $\Hb$, respectively. We also denote by $\mathbb{W}_{\G}\in M(C_u(\G)\ten_{\min} C_u(\h{\G}))$ and $\mathbb{W}_{\Hb}\in M(C_u(\Hb)\ten_{\min} C_u(\h{\Hb}))$ the universal multiplicative unitaries satisfying $(\pi_{\G}\ten\pi_{\h{\G}})(\mathbb{W}_{\G})=W_{\G}$ and $(\pi_{\Hb}\ten\pi_{\h{\Hb}})(\mathbb{W}_{\Hb})=W_{\Hb}$. Finally, we define
$$\W_{\G}:=(\id\ten\pi_{\h{\G}})(\mathbb{W}_{\G})\in M(C_u(\G)\ten_{\min} C_0(\h{\G})), \ \ \ \W_{\Hb}:=(\id\ten\pi_{\h{\Hb}})(\mathbb{W}_{\Hb})\in M(C_u(\Hb)\ten_{\min} C_0(\h{\Hb})).$$
Let $\pi_{\G,\Hb}:C_u(\G)\rightarrow C_u(\Hb)$ be the surjection from \cite[Theorem 3.5]{DKSS}. Its dual morphism is a non-degenerate $*$-homomorphism $\hat{\pi}_{\G,\Hb}:C_u(\h{\Hb})\rightarrow M(C_u(\h{\G}))$ satisfying
$$(\pi_{\G,\Hb}\ten\id)(\mathbb{W}_{\G})=(\id\ten\hat{\pi}_{\G,\Hb})(\mathbb{W}_{\Hb}).$$
From the relation $\gamma\circ\pi_{\wh{\Hb}}=\pi_{\wh{\G}}\circ\hat{\pi}_{\G,\Hb}$ \cite[(3.1)]{DKSS}, we have
\begin{align*}(\id\ten\gamma)(\W_{\Hb})&=(\id\ten\gamma\circ\pi_{\wh{\Hb}})(\mathbb{W}_{\Hb})\\
&=(\id\ten\pi_{\wh{\G}}\circ\hat{\pi}_{\G,\Hb})(\mathbb{W}_{\Hb})\\
&=(\id\ten\pi_{\wh{\G}})(\pi_{\G,\Hb}\ten\id)(\mathbb{W}_{\G})\\
&=(\pi_{\G,\Hb}\ten\id)(\W_{\G}).\end{align*}
Thus, for any $\h{x}\in\LIHH$ and $g\in\LOH$ we have
\begin{align*}\la\h{n},\h{x}\lhd_{\Hb}g\ra&=\la\h{m},\gamma((g\ten \id)W_{\Hb}^*(1\ten\h{x})W_{\Hb})\ra\\
&=\la\h{m},\gamma((\pi_{\Hb}^*(g)\ten \id)(\W_{\Hb}^*(1\ten\h{x})\W_{\Hb}))\ra\\
&=\la\h{m},(\pi_{\Hb}^*(g)\ten \id)(\id\ten\gamma)(\W_{\Hb})^*(1\ten\gamma(\h{x}))(\id\ten\gamma)(\W_{\Hb}))\ra\\
&=\la\h{m},(\pi_{\G,\Hb}^*(\pi_{\Hb}^*(g))\ten \id)(\W_{\G}^*(1\ten\gamma(\h{x}))\W_{\G})\ra\\
&=\la\h{m},\Theta^l(\pi_{\G,\Hb}^*(\pi_{\Hb}^*(g)))(\gamma(\hat{x}))\ra,\end{align*}
where the last equality follows from Lemma \ref{l:lemma}. By invariance of $\h{m}$, we may convolve with any state $f\in\LOQ$ to obtain
\begin{align*}\la\h{n},\h{x}\lhd_{\Hb}g\ra&=\la\h{m},\Theta^l(\pi_{\G,\Hb}^*(g))(\gamma(\hat{x}))\lhd_{\G}f\ra\\
&=\la\h{m},\Theta^l(\pi_{\G,\Hb}^*(\pi_{\Hb}^*(g))\star_{\G}f)(\gamma(\hat{x}))\ra\\
&=\la\h{m},\gamma(\hat{x})\lhd_{\G}(\pi_{\G,\Hb}^*(\pi_{\Hb}^*(g))\star_{\G}f)\ra\\
&=\la\pi_{\G,\Hb}^*(\pi_{\Hb}^*(g))\star_{\G}f,1\ra\la\h{m},\gamma(\hat{x})\ra\\
&=\la g,1\ra\la\h{n},\hat{x}\ra.\end{align*}
\end{proof}

\subsection{Examples arising from the bicrossed product construction} Let $G, G_1$ and $G_2$ be locally compact groups with fixed left Haar measures for which there exists a homomorphism $i:G_1\rightarrow G$ and an anti-homomorphism $j:G_2\rightarrow G$ which have closed ranges and are homeomorphisms onto these ranges. Suppose further that $G_1\times G_2\ni (g,s)\mapsto i(g)j(s)\in G$ is a homeomorphism onto an open subset of $G$ having complement of measure zero. Then $(G_1,G_2)$ is said to be a \textit{matched pair} of locally compact groups \cite[Definition 4.7]{VV}. Any matched pair $(G_1,G_2)$ determines a matched pair of actions $\alpha:G_1\times G_2\rightarrow G_2$ and $\beta:G_1\times G_2\rightarrow G_1$ satisfying mutual co-cycle relations \cite[Lemma 4.9]{VV}. It is known that the von Neumann crossed product $G_1\rtimes_\alpha L^\infty(G_2)$ admits a quantum group structure, called the \textit{bicrossed product} of the matched pair $(G_1,G_2)$. The von Neumann algebra of the dual quantum group is given by the crossed product $L^{\infty}(G_1)^\beta\ltimes G_2$, and therefore, following \cite{Majid}, we denote the bicrossed product quantum group by $VN(G_1)^\beta\bowtie_\alpha L^\infty(G_2)$. 

Below we present sufficient conditions on the matched pair $(G_1,G_2)$ under which the bicrossed product is (strongly) inner amenable. In preparation we collect some useful formulae from \cite[\S4]{VV}, to which we refer the reader for details. To ease the presentation we suppress the notations $i$ and $j$ for the embeddings into $G$.

The fundamental unitary $W$ satisfies
$$W^*\xi(g,s,h,t)=\xi(\beta_{t}(h)^{-1}g,s,h,\alpha_{\beta_t(h)^{-1}g}(s)t), \ \ \ \xi\in L^2(G_1\times G_2\times G_1\times G_2).$$
Letting $\Delta$, $\Delta_1$, and $\Delta_2$ denote the modular functions for the groups $G$, $G_1$, and $G_2$, respectively, the modular conjugation $\h{J}$ of the dual Haar weight satisfies
$$\h{J}\xi(g,s)=\Delta(\alpha_g(s))^{1/2}\Delta_1(\beta_s(g)g^{-1})^{1/2}\Delta_2(\alpha_g(s)s^{-1})^{1/2}\overline{\xi}(\beta_s(g),s^{-1}), \ \ \ \xi\in L^2(G_1\times G_2).$$
Let $\Psi:G_2\times G_1\rightarrow(0,\infty)$ be the (continuous) function determined by the Radon-Nikodym derivatives $\Psi(s,g):=\frac{d\beta_s(g)}{dg}$. It follows that $\Psi(s,g)=\Delta(\alpha_g(s))\Delta_1(\beta_s(g)g^{-1})\Delta_2(\alpha_g(s))$. The action $\beta$ determines a unitary representation of $G_2$ on $L^2(G_1)$ given by
$$v_s\xi(g)=\bigg(\frac{d\beta_{s^{-1}}(g)}{dg}\bigg)^{1/2}\xi(\beta_{s^{-1}}(g)), \ \ \ \xi\in L^2(G_1).$$
We denote by $v^1$ the corresponding action of $G_2$ on $L^1(G_1)$, given by
$$v^1_s\xi(g)=\bigg(\frac{d\beta_{s^{-1}}(g)}{dg}\bigg)f(\beta_{s^{-1}}(g)), \ \ \ f\in L^1(G_1).$$

\begin{prop}\label{p:bicrossed} Let $(G_1,G_2)$ be a matched pair of locally compact groups such that
\begin{enumerate}[label=(\roman*)]
\item $G_2$ is inner amenable,
\item $\beta$ preserves $\Delta_1$,
\item there exists an asymptotically $\beta$-invariant BAI for $L^1(G_1)$, i.e., a net $(f_i)$ of non-negative functions in $C_c(G_1)_{\norm{\cdot}_1=1}$ with $\mathrm{supp}(f_i)\rightarrow\{e\}$ satisfying
$$\norm{v^1_sf_i-f_i}_1\rightarrow0$$
uniformly on compacta.
\end{enumerate}
Then $VN(G_1)^\beta\bowtie_\alpha L^\infty(G_2)$ is strongly inner amenable.
\end{prop}

\begin{proof} Using the co-cycle properties of $\alpha$ and $\beta$ \cite[Lemma 4.9]{VV} together with the definitions of $W$ and $\h{J}$, one sees that
\begin{align*}&(\h{J}\ten\h{J})W^*(\h{J}\ten\h{J})\xi(g,s,h,t)\\
&=\Delta(\alpha_{h^{-1}\beta_s(g)}(s^{-1}))^{1/2}\xi(\beta_{\alpha_{\beta_s(g)}(s^{-1})}(h^{-1})g,s,\beta_{\alpha_{\beta_s(g)}(s^{-1})}(h^{-1})^{-1},t\alpha_{h^{-1}\beta_s(g)}(s^{-1})^{-1})\end{align*}
for all $\xi\in L^2(G_1\times G_2\times G_1\times G_2)$. 
Let $\eta_i:=\sqrt{f_i}$. By hypothesis $(iii)$ we have $\norm{v_s\eta_i - \eta_i}_{L^2(G_1)}^2\leq\norm{v_s^1f_i-f_i}_{L^1(G_1)}\rightarrow0$ 
uniformly on compacta. Combining this with the support condition in $(iii)$, for any uniformly continuous function $f:G_1\times G_2\times G_1\rightarrow\C$ it follows that
\begin{equation}\label{int} \int f(g,s,h) \ (v_{\alpha_{\beta_s(g)}(s^{-1})^{-1}}\eta_i)(h) \overline{\eta_i}(h) \ dh \rightarrow f(g,s,e)\end{equation}
uniformly for $(g,s)$ in compact subsets $K\subseteq G_1\times G_2$.

By $(i)$ there exits a net $(\xi_j)$ in $C_c(G_2)_{\norm{\cdot}_2=1}$ satisfying
\begin{equation}\label{e:conv}\norm{\rho(s)\xi_j-\lm(s)^*\xi_j}_{L^2(G_2)}\rightarrow0\end{equation}
uniformly on compacta. Let $U_1$ be the self-adjoint unitary on $L^2(G_1)$ satisfying $U_1\xi(g)=\xi(g^{-1})\Delta_1(g^{-1})$. Then for any $\eta\in C_c(G_1\times G_2)$, and any $j$, we have
\begin{align*}&\la W\sigma V\sigma(\eta\ten U_1\eta_i\ten\xi_j),\eta\ten U_1\eta_i\ten\xi_j\ra=\la(\h{J}\ten\h{J})W^*(\h{J}\ten\h{J})(\eta\ten U_1\eta_i\ten\xi_j),W^*(\eta\ten U_1\eta_i\ten\xi_j)\ra\\
&=\iiiint\Delta(\alpha_{h^{-1}\beta_s(g)}(s^{-1}))^{1/2} \ \eta(\beta_{\alpha_{\beta_s(g)}(s^{-1})}(h^{-1})g,s) \ U_1\eta_i(\beta_{\alpha_{\beta_s(g)}(s^{-1})}(h^{-1})^{-1}) \ \xi_j(t(\alpha_{h^{-1}\beta_s(g)}(s^{-1}))^{-1})\\
&\times \overline{\eta}(\beta_t(h)^{-1}g,s) \ \overline{U_1\eta_i}(h) \ \overline{\xi_j}(\alpha_{\beta_t(h)^{-1}g}(s)t) \ dg \ ds \ dh \ dt\\
&=\iiiint\Delta(\alpha_{h^{-1}\beta_s(g)}(s^{-1}))^{1/2} \ \eta(\beta_{\alpha_{\beta_s(g)}(s^{-1})}(h^{-1})g,s) \ \eta_i(\beta_{\alpha_{\beta_s(g)}(s^{-1})}(h^{-1})) \ \xi_j(t(\alpha_{h^{-1}\beta_s(g)}(s^{-1}))^{-1})\\
&\times \overline{\eta}(\beta_t(h)^{-1}g,s) \ \overline{\eta_i}(h^{-1}) \ \overline{\xi_j}(\alpha_{\beta_t(h)^{-1}g}(s)t) \ \Delta_1(h^{-1}) \ dg \ ds \ dh \ dt\\
&=\iiiint\Delta(\alpha_{h^{-1}\beta_s(g)}(s^{-1}))^{1/2} \ \eta(\beta_{\alpha_{\beta_s(g)}(s^{-1})}(h^{-1})g,s) \ (v_{\alpha_{\beta_s(g)}(s^{-1})^{-1}}\eta_i)(h^{-1}) \ \xi_j(t(\alpha_{h^{-1}\beta_s(g)}(s^{-1}))^{-1})\\
&\times \Psi(\alpha_{\beta_s(g)}(s^{-1}),h^{-1})^{-1/2} \ \overline{\eta}(\beta_t(h)^{-1}g,s) \ \overline{\eta_i}(h^{-1}) \ \overline{\xi_j}(\alpha_{\beta_t(h)^{-1}g}(s)t) \ \Delta_1(h^{-1}) \ dg \ ds \ dh \ dt\\
&=\iiiint\Delta(\alpha_{h\beta_s(g)}(s^{-1}))^{1/2} \ \eta(\beta_{\alpha_{\beta_s(g)}(s^{-1})}(h)g,s) \ (v_{\alpha_{\beta_s(g)}(s^{-1})^{-1}}\eta_i)(h) \ \xi_j(t(\alpha_{h\beta_s(g)}(s^{-1}))^{-1})\\
&\times \Psi(\alpha_{\beta_s(g)}(s^{-1}),h)^{-1/2} \ \overline{\eta}(\beta_t(h^{-1})^{-1}g,s) \ \overline{\eta_i}(h) \ \overline{\xi_j}(\alpha_{\beta_t(h^{-1})^{-1}g}(s)t) \ dg \ ds \ dh \ dt\\
&=\iiint \bigg(\int \xi_j(t(\alpha_{h\beta_s(g)}(s^{-1}))^{-1}) \ \overline{\xi_j}(\alpha_{\beta_t(h^{-1})^{-1}g}(s)t) \ \overline{\eta}(\beta_t(h^{-1})^{-1}g,s) \ dt\bigg)\\
&\times \Delta(\alpha_{h\beta_s(g)}(s^{-1}))^{1/2} \ \eta(\beta_{\alpha_{\beta_s(g)}(s^{-1})}(h)g,s) \ \Psi(\alpha_{\beta_s(g)}(s^{-1}),h)^{-1/2} \   (v_{\alpha_{\beta_s(g)}(s^{-1})^{-1}}\eta_i)(h) \overline{\eta_i}(h) \ dg \ ds \ dh\\
&\rightarrow\iiint \xi_j(t(\alpha_{\beta_s(g)}(s^{-1}))^{-1})\  \overline{\xi_j}(\alpha_{g}(s)t) \ |\eta(g,s)|^2 \ 
\Delta(\alpha_{\beta_s(g)}(s^{-1}))^{1/2} \ \Psi(\alpha_{\beta_s(g)}(s^{-1}),e)^{-1/2} \ dg \ ds \ dt\\
\end{align*}
by (\ref{int}). But $\alpha_{\beta_s(g)}(s^{-1})^{-1}=\alpha_g(s)$ almost everywhere in $(g,s)$, so that
\begin{align*}\Delta(\alpha_{\beta_s(g)}(s^{-1}))^{1/2}\Psi(\alpha_{\beta_s(g)}(s^{-1}),e)^{-1/2}&=\Delta(\alpha_{\beta_s(g)}(s^{-1}))^{1/2} \Delta(\alpha_{\beta_s(g)}(s^{-1}))^{-1/2}\Delta_2(\alpha_{\beta_s(g)}(s^{-1}))^{-1/2}\\
&=\Delta_2(\alpha_g(s))^{1/2}, \ \ \ a.e.
\end{align*}
The final integral in the above calculation therefore reduces to
$$\iiint \Delta_2(\alpha_g(s))^{1/2} \ \xi_j(t\alpha_{g}(s))\  \overline{\xi_j}(\alpha_{g}(s)t) \ |\eta(g,s)|^2  \ dg \ ds \ dt.$$
By the compact convergence (\ref{e:conv}) the above expression converges in $j$ to
$$\iint |\eta(g,s)|^2 \ dg \ ds = \norm{\eta}_{L^2(G_1\times G_2)}^2.$$
Denoting the index sets of $(\eta_i)$ and $(\xi_j)$ by $I$ and $J$, respectively, we form the product $\mc{I}:=J\times I^{J}$ and for $I=(j,(i_j)_{j\in J})\in\mc{I}$ we let $\xi_I:=U_1\eta_{i_j}\ten\xi_j\in L^2(G_1\times G_2)$. By \cite[pg. 69]{Kelley} and the above analysis, the resulting net $(\xi_I)$ is asymptotically conjugation invariant for $VN(G_1)^\beta\bowtie_\alpha L^\infty(G_2)$.
\end{proof}

\begin{cor}\label{c:discrete} Let $(G_1,G_2)$ be matched pair of discrete groups. Then $VN(G_1)^\beta\bowtie_\alpha L^\infty(G_2)$ is strongly inner amenable.\end{cor}

\begin{remark} Unlike amenability, inner amenability for locally compact groups does not pass to extensions. For example, $\R^2\rtimes \F_6$ is not inner amenable (see Remark \ref{r:6}), while both $\R^2$ and $\F_6$ are. \end{remark}

\begin{example} We consider a discretized version of \cite[Example 1]{DQV}. Let
$$G=\bigg\{\begin{pmatrix} a & b & x\\ c & d & y\\ 0 & 0 & 1\end{pmatrix}\mid\begin{pmatrix}a & b\\ c & d\end{pmatrix}\in SL(2,\Q), \ x,y\in\Q\bigg\},$$
$G_1=(\Q^2,+)$, and $G_2=SL(2,\Q)$, viewed as discrete groups. The embeddings
$$G_1\ni (x,y)\mapsto \begin{pmatrix} 1 & 0 & -x\\ -x & 1 & -y+\frac{1}{2}x^2\\ 0 & 0 & 1\end{pmatrix}\in G, \ \ G_2\ni \begin{pmatrix}a & b\\ c & d\end{pmatrix}\mapsto  \begin{pmatrix}d & -b & 0\\ -c & a & 0\\ 0 & 0 & 1\end{pmatrix}\in G$$
determine a matched pair structure for $(G_1,G_2)$. The corresponding actions are given by
\begin{align*}\alpha_{(x,y)}\begin{pmatrix}a & b\\ c & d\end{pmatrix}&=\begin{pmatrix}a+bx & b\\ c+dx-(a+bx)(ax+b(y+\frac{1}{2}x^2)) & d-b(ax+b(y+\frac{1}{2}x^2))\end{pmatrix} \\
\beta_{\begin{pmatrix}a & b\\ c & d\end{pmatrix}}(x,y)&=\bigg(ax+by+\frac{b}{2}x^2,cx+d\bigg(y+\frac{1}{2}x^2\bigg)-\frac{1}{2}\bigg(ax+b\bigg(y+\frac{1}{2}x^2\bigg)\bigg)^2\bigg).\end{align*}
By Corollary \ref{c:discrete}, the bicrossed product $\G:=VN(\Q^2)^\beta\bowtie_\alpha L^\infty(SL(2,\Q))$ is strongly inner amenable. Since $SL(2,\Q)$ is not amenable, it follows from \cite[Theorem 13]{DQV} that $\G$ is not amenable. Moreover, since $\Q^2$ is not compact, $\G$ is not discrete by \cite[Proposition 2.17]{VV}. Thus, $\G$ is an example of a non-discrete, non-amenable, inner amenable quantum group.

\end{example}

\section{Relative Injectivity and the Averaging Technique}

In \cite[Lemma 4.1]{RX} Ruan and Xu showed that the dual $\LIQH$ of a strongly inner amenable Kac algebra $\G$ is relatively $1$-injective in $\LOQH\hskip2pt\mathbf{mod}$. We will now show that relative $1$-injectivity follows from the a priori weaker notion of inner amenability. Below we let $\widetilde{\G}:=\mathrm{Gr}(\h{\G})$ denote the intrinsic group of $\h{\G}$. 

\begin{prop}\label{p:IA=>relinj} Let $\G$ be a locally compact quantum group. Consider the following conditions:
\begin{enumerate}[label=(\roman*)]
\item $\h{\G}$ is inner amenable;
\item $\LIQ$ is relatively $1$-injective in $\mathbf{mod}\hskip2pt\LOQ$;
\item $\LIQ$ is relatively $1$-injective in $\LOQ\hskip2pt\mathbf{mod}$;
\item $\widetilde{\h{\G}}=\mathrm{Gr}(\G)$ is inner amenable.
\end{enumerate}
Then $(i)\Rightarrow(ii)\Leftrightarrow(iii)\Rightarrow(iv)$. When $\G$ is co-commutative, the conditions are equivalent.
\end{prop}

\begin{proof} $(i)\Rightarrow(ii)$: Given a state $n\in\LIQ^*$ which is right $\widehat{\lhd}$ invariant, it follows that $m:=n\circ R$ is left $\widehat{\rhd}'$ invariant.

It suffices to provide a completely contractive morphism which is a left inverse to the map $\Delta:\LIQ\rightarrow\mc{CB}(\LOQ,\LIQ)$ given by
\begin{equation}\label{delta}\Delta(x)(f)=x\rhd f, \ \ \ T\in\BLTQ, \ \rho\in\TCQ.\end{equation}
Identifying $\mc{CB}(\LOQ,\LIQ)\cong\LIQ\oten\LIQ$ via
\begin{equation*}\la\Phi,f\ten g\ra=\la\Phi(f),g\ra, \ \ \ \Phi\in\mc{CB}(\LOQ,\LIQ), \ f,g\in\LOQ,\end{equation*}
we have $\Delta=\Gam$, and that the corresponding $\LOQ$-module structure on $\LIQ\oten\LIQ$ is defined by $X\unrhd f=(f\ten\id\ten\id)(\Gam^r\ten\id)(X)$ for $X\in\LIQ\oten\LIQ$ and $f\in\LOQ$.

First, consider the map 
$$\Phi:\BLTQ\oten\LIQ\ni A\mapsto (\id\ten m)(V^*AV)\in\BLTQ.$$
Clearly, $\Phi$ is a completely contractive left inverse to $\Gam$. We show that $\Phi$ is a right $\mc{T}_\rhd$-module map. This will complete the proof since \cite[Corollary 4.3]{C} will entail the invariance $\Phi(\LIQ\oten\LIQ)\subseteq\LIQ$, and the restricted module action $\mc{T}_\rhd\curvearrowright\LIQ$ is the pertinent $\LOQ$-module action. To this end, fix $A\in\BLTQ\oten\LIQ$ and $\rho\in\TCQ$. Then
\begin{align*}\Phi(A\unrhd\rho)&=\Phi((\rho\ten\id\ten\id)(V_{12}A_{13}V_{12}^*))\\
&=(\id\ten m)(\rho\ten\id\ten\id)(V_{23}^*V_{12}A_{13}V_{12}^*V_{23})\\
                              &=(\id\ten m)(\rho\ten\id\ten\id)(V_{12}V_{23}^*V_{13}^*A_{13}V_{13}V_{23}V_{12}^*)\\
                              &=(\rho\ten\id)(V(\id\ten\id\ten  m)(V_{23}^*V_{13}^*A_{13}V_{13}V_{23})V^*).\end{align*}
Now, using the fact that $\h{V}'=\sigma V^*\sigma$, where $\sigma$ is the flip map on $\LTQ\ten\LTQ$, for any $\tau,\om\in\TCQ$, we have
\begin{align*}&\la(\id\ten\id\ten  m)(V_{23}^*V_{13}^*A_{13}V_{13}V_{23}),\tau\ten\om\ra\\
&=\la(\id\ten\id\ten  m)(V_{23}^*(\sigma\ten 1)V_{23}^*A_{23}V_{23}(\sigma\ten1)V_{23}),\tau\ten\om\ra\\
&=\la(\id\ten\id\ten  m)(V_{13}^*V_{23}^*A_{23}V_{23}V_{13}),\om\ten\tau\ra\\
&=\la(\id\ten  m)(V^*(1\ten(\tau\ten\id)(V^*AV))V),\om\ra\\
&=\la( m\ten\id)(\h{V}'((\tau\ten\id)(V^*AV)\ten1)\h{V}'^*),\om\ra\\
&=\la  m,\om\wh{\rhd}'((\tau\ten\id)(V^*AV))\ra\\
&=\la  m,(\tau\ten\id)(V^*AV)\ra\la\om,1\ra\\
&=\la(\id\ten  m\ten\id)(V^*AV\ten 1),\tau\ten\om\ra\\
&=\la\Phi(A)\ten 1,\tau\ten\om\ra.\end{align*}

As $\tau$ and $\om$ were arbitrary, we have
\begin{align*}\Phi(A\unrhd\rho)&=(\rho\ten\id)(V(\id\ten\id\ten m)(V_{23}^*V_{13}^*A_{13}V_{13}V_{23})V^*)\\
                             &=(\rho\ten\id)(V(\Phi(A)\ten1)V^*)\\
                             &=\Phi(A)\rhd\rho.\end{align*}
                             
$(ii)\Leftrightarrow (iii)$ Given a completely contractive left (respectively, right) $\LOQ$-module left inverse $\Phi$ to $\Gam$, it follows that $R\circ\Phi\circ\Sigma\circ(R\ten R)$ is a completely contractive right (respectively, left) $\LOQ$-module left inverse to $\Gam$. 
                          
$(iii)\Rightarrow(iv)$: Recall that $\mathrm{Gr}(\G)$ is a group of unitaries in $\LIQ$, so it acts naturally on $\LIQ$ by conjugation. The existence of a state $m\in\LIQ^*$ which is $\mathrm{Gr}(\G)$-invariant follows directly from the argument of \cite[Theorem 3.4]{CT}, using the $\Gam(\LIQ)-\LIQ$-bimodule property of $\Phi$. Since $\widetilde{\widehat{\G}}$ is a closed quantum subgroup of $\widehat{\G}$ in the sense of Vaes \cite[Theorem 5.5]{D2}, there exists a normal $*$-homomorphism $\gamma:L^{\infty}(\widehat{\widetilde{\h{\G}}})\rightarrow\LIQ$ intertwining the co-multiplications. As $L^{\infty}(\widehat{\widetilde{\h{\G}}})=VN(\widetilde{\widehat{\G}})=VN(\mathrm{Gr}(\G))$, the state $m\circ\gamma\in VN(\mathrm{Gr}(\G))^*$ is $\mathrm{Gr}(\G)$-invariant, making $\mathrm{Gr}(\G)$ inner amenable by \cite[Proposition 3.2]{CT}.

When $\G=\G_s$ is co-commutative, then $\mathrm{Gr}(\G_s)=G$ and the implication $(iv)\Rightarrow(i)$ follows immediately from \cite[Proposition 3.2]{CT}.
\end{proof}

\begin{cor}\label{c:equiv} Let $\G$ be a locally compact quantum group for which $\LIQH$ is an injective von Neumann algebra. Then the following are equivalent:
\begin{enumerate}[label=(\roman*)]
\item $\G$ is amenable;
\item $\G$ is inner amenable;
\item $\LIQH$ is relatively 1-injective in $\mathbf{mod}\hskip2pt\LOQH$.
\end{enumerate}
\end{cor}

\begin{proof} Propositions \ref{p:amen=>IA} and \ref{p:IA=>relinj} yield the implications $(i)\Rightarrow(ii)\Rightarrow(iii)$. Assume $(iii)$. Since $\LIQH$ is 1-injective in $\mathbf{mod}\hskip2pt\C$ it follows from \cite[Proposition 2.3]{C} that $\LIQH$ is 1-injective in $\mathbf{mod}\hskip2pt\LOQH$. Hence, $\G$ is amenable by \cite[Theorem 5.1]{C}.
\end{proof}


In the recent article \cite{NV}, Ng and Viselter utilized topological inner amenability to elucidate the connection between co-amenability of $\G$ and amenability of $\h{\G}$. One of their main results is the following.

\begin{thm}[Ng--Viselter]\label{t:Ng-Viselter} Let $\G$ be a locally compact quantum group.
\begin{enumerate}[label=(\roman*)]
\item Consider the following conditions:
\begin{enumerate}
\item $\G$ is co-amenable;
\item $C_0(\G)$ is nuclear and there exists a state $\rho\in C_0(\G)^*$ such that
$$\la \rho,x\widehat{\lhd} \hat{f}\ra=\la\rho,x\ra\la \hat{f},1\ra, \ \ \ \hat{f}\in\LOQH, \ x\in C_0(\G).$$
\item $\h{\G}$ is amenable.
\end{enumerate}
Then $(a)\Rightarrow(b)\Rightarrow(c)$.
\item Moreover, if $\h{\G}$ has trivial scaling group (for instance, if $\G$ is a Kac algebra), then $(a)\Rightarrow(b')\Rightarrow(b)$, where
\begin{enumerate}
\item [(b$'$)] $C_0(\G)$ is nuclear and has a tracial state.
\end{enumerate}
\end{enumerate}
\end{thm}

They conjectured that condition $(b)$ above is equivalent to either condition $(a)$ or condition $(c)$. We now show that  $(b)$ is indeed equivalent to $(a)$.

\begin{thm}\label{t:co-amen} Let $\G$ be a locally compact quantum group. Then $\G$ is co-amenable if and only if $C_0(\G)$ is nuclear and $\h{\G}$ is topologically inner amenable.\end{thm}

\begin{proof} Suppose that $C_0(\G)$ is nuclear and $\h{\G}$ is topologically inner amenable.  Given a state $m\in C_0(\G)^*$ which is right $\widehat{\lhd}$ invariant, it follows that $n:=m\circ R$ is left $\widehat{\rhd}'$ invariant, as in the proof of Proposition \ref{p:IA=>relinj}. Let $n\in M(C_0(\G))^*$ also denote the unique extension to $M(C_0(\G))$ which is strictly continuous on the unit ball, and $(\id\ten n):M(\KLTQ\ten_{\min} C_0(\G))\rightarrow \BLTQ$ denote the unique extension of the slice map $(\id\ten n):\KLTQ\ten_{\min} C_0(\G)\rightarrow \KLTQ$ which is strictly continuous on the unit ball. By strict density of $C_0(\G)$ in $M(C_0(\G))$, it follows that
$$\la n,\hat{f}'\widehat{\rhd}'x\ra=\la n,x\ra\la \hat{f}',1\ra, \ \ \ \hat{f}\in\LOQHP, \ x\in M(C_0(\G)).$$
Since $V\in M(\KLTQ \ten_{\min} C_0(\G))$, the map $\Phi:C_0(\G)\ten_{\min} C_0(\G)\rightarrow\BLTQ$ defined by
$$\Phi(A)=(\id\ten n)(V^*AV), \ \ \ A\in C_0(\G)\ten_{\min} C_0(\G),$$
is a non-zero, strict, completely positive contraction. Using the extended $\widehat{\rhd}'$-invariance on $M(C_0(\G))$, it follows verbatim from the proof of Proposition \ref{p:IA=>relinj} that
$$\Phi(A\unrhd\rho)=\Phi(A)\rhd\rho, \ \ \ A\in C_0(\G)\ten_{\min} C_0(\G), \ \rho\in\TCQ.$$
Since $\mathrm{Ad}(V^*):\KLTQ \ten_{\min} C_0(\G)\rightarrow\KLTQ \ten_{\min} C_0(\G)$ and
$$\TCQ\hten M(\G)=(\KLTQ \ten_{\min} C_0(\G))^*,$$
we have 
$$\mathrm{Ad}(V^*)^*:\TCQ\hten M(\G)\rightarrow\TCQ\hten M(\G).$$ 
Letting $r:\TCQ\ni\rho\mapsto\rho|_{C_0(\G)}\in M(\G)$ be the (completely positive) restriction map, it follows that
$$\Phi^*|_{\TCQ}:\TCQ\ni\rho\mapsto(r\ten \id)(\mathrm{Ad}(V^*)^*(\rho\ten n))\in M(\G)\hten M(\G).$$
The proof of \cite[Corollary 4.3]{C} entails the inclusion $\Phi(C_0(\G)\ten_{\min} C_0(\G))\subseteq\LIQ$. In fact, since $\Phi$ is a right $\LOQ$-module map and $C_0(\G)$ is essential, more is true:
$$\Phi(C_0(\G)\ten_{\min} C_0(\G))\subseteq\LUC\subseteq M(C_0(\G)).$$
The unique strict extension $\widetilde{\Phi}:M(C_0(\G)\ten_{\min} C_0(\G))\rightarrow M(C_0(\G))$, which exists by \cite[Corollary 5.7]{Lance}, satisfies $\widetilde{\Phi}\circ\Gamma|_{C_0(\G)}=\id_{C_0(\G)}$. 
If $\rho\in\LIQ_{\perp}$ then the invariance $\Phi(C_0(\G)\ten_{\min} C_0(\G))\subseteq\LIQ$ implies $\Phi^*(\rho)=0$. Thus, $\Phi^*$ induces a completely positive left $\LOQ$-module map 
$$\Phi^*:\LOQ=(\TCQ/\LIQ_{\perp})\rightarrow M(\G)\hten M(\G).$$
By strict continuity and the definition of the multiplication on $M(\G)$, for $f\in\LOQ$ and $x\in C_0(\G)$ we have
$$\la m_{M(\G)}(\Phi^*(f)),x\ra=\la\Phi^*(f),\Gam(x)\ra=\la f,\widetilde{\Phi}(\Gam(x))\ra=\la f,x\ra.$$
Hence, $m_{M(\G)}\circ\Phi^*$ is the canonical inclusion $\LOQ\hookrightarrow M(\G)$. 


Now, by nuclearity of $C_0(\G)$, there exists a net $\vphi_i:C_0(\G)\rightarrow C_0(\G)$ of finite-rank completely positive contractions converging to the identity in the point-norm topology. Define $\psi_i:\LOQ\rightarrow M(\G)$ by
$$\psi_i=m_{M(\G)}\circ(\id\ten\vphi_i^*)\circ \Phi^*.$$
Then $\psi_i\in \ _{\LOQ}\mc{CP}(\LOQ,M(\G))=_{\LOQ}\mc{CP}(\LOQ)$. By \cite[Theorem 5.2]{D3} there exist contractive positive functionals $\nu_i\in C_u(\G)^*$ such that $\psi_i=m^r_{\nu_i}$. Since $\vphi^*_i:M(\G)\rightarrow M(\G)$ forms a bounded net converging to the identity point-weak*, and $m_{M(\G)}$ is separately weak* continuous, it follows that $\psi_i$ converges to the inclusion $\LOQ\hookrightarrow M(\G)$ point-weak*.

Write $\vphi_i=\sum_{k=1}^{n_i}x_k^i\ten \mu_k^i$ for some $x_k^i\in C_0(\G)$ and $\mu_k^i\in M(\G)$.  Since $(b)\Rightarrow(c)$ in Theorem \ref{t:Ng-Viselter} we know that $\h{\G}$ is amenable. Hence, by \cite[Proposition 5.10]{C} $M^r_{cb}(\LOQ)= _{\LOQ}\mc{CB}(\LOQ)=C_u(\G)^*$. For each $i$ and $k$ the map $(\id\ten x_k^i)\Phi^*\in _{\LOQ}\mc{CB}(\LOQ)$, so there exists $\nu_k^i$ such that $(\id\ten x_k^i)\Phi^*=m^r_{\nu_k^i}$. Thus, for $f\in\LOQ$ we have
\begin{align*}\psi_i(f)&=m_{M(\G)}\circ(\id\ten\vphi_i^*)\circ \Phi^*(f)\\
&=\sum_{k=1}^{n_i}m_{M(\G)}((\id\ten x_k^i)\Phi^*(f)\ten\mu_k^i)\\
&=\sum_{k=1}^{n_i}m_{M(\G)}(f\star\nu_k^i\ten\mu_k^i)\\
&=\sum_{k=1}^{n_i}f\star\nu_k^i\star\mu_k^i.\end{align*}
Hence, $\nu_i=\sum_{k=1}^{n_i}\nu_k^i\star\mu_k^i\in M(\G)$ as $M(\G)$ is a closed ideal in $C_u(\G)^*$. Passing to a subnet we may assume $\nu_i\rightarrow\nu$ weak* in $M(\G)$. But then for any $f\in\LOQ$ and $x\in C_0(\G)$ we have
$$\la f\star\nu, x\ra=\la\nu, x\star f\ra=\lim_i\la\nu_i,x\star f\ra=\lim_i\la f\star\nu_i,x\ra=\la f,x\ra.$$
It follows that $\nu$ is a right identity for $M(\G)$, which implies that $M(\G)$ is unital, whence $\G$ is co-amenable (cf. \cite[Theorem 3.1]{BT}).
\end{proof}

Theorem \ref{t:co-amen} generalizes, and provides a different proof of, the main result in \cite{Ng}, which says that a locally compact group $G$ is amenable if and only if $C^*_\lm(G)$ is nuclear and has a tracial state.

\begin{cor} Let $\G$ be a locally compact quantum group such that $\h{\G}$ has trivial scaling group (for instance, if $\G$ is a Kac algebra). Then $\G$ is co-amenable if and only if $C_0(\G)$ is nuclear and has a tracial state.\end{cor}

The combination of Theorem \ref{t:co-amen} with Corollary \ref{c:equiv} elucidates the relationship between co-amenability and amenability of the dual: $\h{\G}$ is amenable if and only if $C_0(\G)$ is nuclear and there exists a state $m\in\LIQ^*$ such that
$$\la m,x\widehat{\lhd} \h{f}\ra=\la m,x\ra\la \hat{f},1\ra, \ \ \ \hat{f}\in\LOQH, \ x\in \LIQ,$$
while $\G$ is co-amenable if and only if $C_0(\G)$ is nuclear and there exists a state $m\in C_0(\G)^*$ such that
$$\la m,x\widehat{\lhd} \h{f}\ra=\la m,x\ra\la \hat{f},1\ra, \ \ \ \hat{f}\in\LOQH, \ x\in C_0(\G).$$
This subtle difference can also be phrased in terms of homology: $\h{\G}$ is amenable if and only if $\LOQ$ is 1-flat \cite[Theorem 5.1]{C}, while $\G$ is co-amenable if and only if $M(\G)$ is 1-projective, as we now prove.

\begin{thm} Let $\G$ be a locally compact quantum group. Then $\G$ is co-amenable if and only if $M(\G)$ is 1-projective in $M(\G)\hskip2pt\mathbf{mod}$.\end{thm}

\begin{proof} If $\G$ is co-amenable then $M(\G)$ is unital by \cite[Theorem 3.1]{BT}, and moreover the unit has norm one. Any unital completely contractive Banach algebra $\mc{A}$ is $\norm{e_{\mc{A}}}$-projective, so the claim follows.

Conversely, if $M(\G)$ is 1-projective, then $C_0(\G)^{**}=M(\G)^*$ is 1-injective in $\mathbf{mod}-M(\G)$. It follows that the inclusion morphism $M(C_0(\G))\hookrightarrow C_0(\G)^{**}$ extends to a unital completely contractive, hence completely positive $M(\G)$-module map $\Phi: \BLTQ\rightarrow C_0(\G)^{**}$. Applying the argument of \cite[Theorem 3.2]{NV} we see that $(\id\ten\Phi):C_0(\G)^{**}\oten\BLTQ\rightarrow C_0(\G)^{**}\oten C_0(\G)^{**}$ is a unital completely positive morphism satisfying $(\id\ten \Phi)(\h{W})=\h{W}$. By unitarity it follows that $\h{W}$ is in the multiplicative domain of $(\id\ten\Phi)$, and hence
$$(\id\ten\Phi)(\h{W}^*X\h{W})=\h{W}^*(\id\ten\Phi)(X)\h{W}, \ \ \ X\in C_0(\G)^{**}\oten\BLTQ.$$
In particular, for every $T\in\BLTQ$ and $\hat{f}\in\LOQH$ we have
\begin{align*}\Phi(T\h{\lhd}\hat{f})&=\Phi((\hat{f}\ten\id)(\h{W}^*(1\ten T)\h{W}))=(\hat{f}\ten\id)(\id\ten\Phi)((\h{W}^*(1\ten T)\h{W}))\\
&=(\hat{f}\ten\id)(\h{W}^*(1\ten \Phi(T))\h{W})=\Phi(T)\h{\lhd}\hat{f}.\end{align*}
Thus, $\Phi$ is a morphism with respect to the canonical $\LOQH$-module structure on $C_0(\G)^{**}$. Moreover, the $M(\G)$-module property entails $\pi\circ\Phi(\LIQH)=\C1$, where $\pi:C_0(\G)^{**}\rightarrow\LIQ$ is the canonical surjection. Since $\pi$ is clearly an $\LOQH$-module map, it follows that $\pi\circ\Phi|_{\LIQH}$ is an invariant mean, entailing the amenability of $\h{\G}$ and therefore the nuclearity of $C_0(\G)$. 

Now, by 1-projectivity of $M(\G)$, for every $\ep>0$ there exists a morphism $\Phi_\ep:M(\G)\rightarrow M(\G)^+\hten M(\G)$ such that $m^+\circ \Phi_\ep=\id_{M(\G)}$, and $\norm{\Phi_\ep}_{cb}<1+\ep$. By amenability of $\h{\G}$, we know $C_u(\G)^*= _{\LOQ}\mc{CB}(\LOQ)$ \cite[Proposition 5.10]{C}, so one does not require the complete positivity of $\Phi_\ep$ to perform the averaging argument from Theorem \ref{t:co-amen}, which yields a bounded net $(\mu_i)$ in $M(\G)$ that clusters to a right identity, entailing the co-amenability of $\G$.
\end{proof}

The averaging argument used above, together with its variants used in \cite{Ar2,C}, shows that it is \textit{inner amenability}, as opposed to discreteness, that underlies the original averaging technique of Haagerup. In the setting of unimodular discrete quantum groups $\h{\G}$, the technique relies on the existence of a normal left inverse $\Phi:\LIQ\oten \LIQ\rightarrow\LIQ$ to the co-multiplication that is an $\LOQ$-module map. Such a map is typically built from a trace-preserving normal conditional expectation $E:\LIQ\oten \LIQ\rightarrow \Gam(\LIQ)$ onto the image of $\Gam$ (see \cite[Theorem 7.5]{DFSW} and \cite[Theorem 5.5]{KR}). It is the combination of the $\LOQ$-module property of $\Phi$ together with a suitable finite-dimensional approximation that allows one to average approximation properties of $\LIQ$ or $C_0(\G)$ to approximation properties of $\h{\G}$. Thus, provided one has a suitably nice $\LOQ$-module left inverse to the co-multiplication, the same averaging technique applies. This is where inner amenability enters the picture.

Recall that a locally compact quantum group $\G$ is \textit{weakly amenable} if there exists an approximate identity $(\hat{f}_i)$ in $\LOQH$ which is bounded in $\McbQHl$. The infimum of bounds for such approximate identities is the Cowling--Haagerup constant of $\G$, and is denoted $\Lambda_{cb}(\G)$. We say that $\G$ has the \textit{approximation property} if there exists a net $(\hat{f}_i)$ in $\LOQH$ such that $\h{\Theta}^l(\hat{\lm}(\hat{f}_i))$ converges to $\id_{\LIQH}$ in the stable point-weak* topology.

\begin{prop}\label{p:AP} Let $\G$ be a locally compact quantum group whose dual $\h{\G}$ is strongly inner amenable.

\begin{enumerate}[label=(\roman*)]
\item If $\LIQ$ has the w*CBAP then $\h{\G}$ is weakly amenable with $\Lambda_{cb}(\h{\G})\leq\Lambda_{cb}(\LIQ)$.
\item $\LIQ$ has the w*OAP if and only if $\h{\G}$ has the approximation property.
\end{enumerate}
\end{prop}

\begin{proof} Let $(\xi_i)$ be a net of asymptotically conjugation invariant unit vectors in $\LTQ$. It follows verbatim from \cite[Lemma 4.1]{RX} that 
$$\Phi_i:\LIQ\oten\LIQ\ni X\mapsto (\om_{\xi_i}\ten\id)(W^*(U^*\ten 1)X(U\ten 1)W)\in\LIQ.$$
defines a net of unital completely positive left $\LOQ$-module maps, which cluster weak* in 
$$\mc{CB}(\LIQ\oten\LIQ,\LIQ)=((\LIQ\oten\LIQ)\hten\LOQ)^*$$
to a module left inverse to $\Gamma$. Passing to a subnet we may assume convergence. 

$(i)$ Let $(\vphi_j)$ be a net of finite-rank, normal, completely bounded maps converging to the identity point-weak*, with $\norm{\vphi_j}_{cb}\leq C$. Since $\vphi_j$ is finite-rank, there exists $f^j_1,...,f^j_{n_j}\in\LOQ$ and $x^j_1,...,x^j_{n_j}\in\LIQ$ such that $\vphi_j=\sum_{k=1}^{n_j} x^j_k\ten f^j_k$. Put
$$\Phi_{ij}=\Phi_i\circ(\vphi_j\ten\id)\circ\Gam:\LIQ\rightarrow\LIQ.$$
Then $\Phi_{ij}$ is a normal completely bounded left $\LOQ$-module map with $\norm{\Phi_{ij}}_{cb}\leq C$. Also, for each $i,j$ and $1\leq k\leq n_j$, the map
\begin{equation}\label{e:e1}\LIQ\ni x\mapsto \Phi_i(x^j_k\ten x)\in\LIQ\end{equation}
is a normal completely bounded left $\LOQ$-module map. Since $x^j_k$ is a linear combination of positive elements in $\LIQ$, and $\Phi_i$ is completely positive, the map (\ref{e:e1}) is a linear combination of normal completely positive $\LOQ$-module maps $\LIQ\rightarrow\LIQ$. By \cite[Theorem 5.2]{D3}, there exist $\mu^{ij}_k\in C_u(\G)^*$ such that (\ref{e:e1}) is given by right multiplication by $\mu^{ij}_k$. Hence, for each $x\in \LIQ$
\begin{align*}\Phi_{ij}(x)&=\Phi_i\circ(\vphi_j\ten\id)\circ\Gam(x)\\
&=\sum_{k=1}^{n_j}\Phi_i(x^j_k \ten x\star f^j_k)\\
&=\sum_{k=1}^{n_j}x\star f^j_k\star\mu^{ij}_k\\
&=x\star f_{ij},\end{align*}
where $f_{ij}=\sum_{k=1}^{n_i} f^j_k\star\mu^{ij}_k\in\LOQ$ as $\LOQ$ is a closed ideal in $C_u(\G)^*$. Thus, $\norm{f_{ij}}_{cb}\leq C$, and it follows that
$$\lim_i\lim_jx\star f_{ij}=\lim_i\Phi_i(\Gamma(x))=x,$$
point-weak* in $\LIQ$. Combining the iterated limit into a single net as in Proposition \ref{p:bicrossed} and appealing to the standard convexity argument, it follows that $\h{\G}$ is weakly amenable with $\Lambda_{cb}(\h{\G})\leq\Lambda_{cb}(\LIQ)$.

$(ii)$ The proof that w*OAP  implies the approximation property follows by a similar argument to $(i)$, appealing to well-known properties of the stable point-weak* topology on von Neumann algebras (cf. \cite[Proposition 1.7]{HK}). The converse was shown for Kac algebras in \cite[Theorem 4.15]{KR}. Their proof extends verbatim to arbitrary locally compact quantum groups. 
\end{proof}

\begin{cor} Let $G$ be an inner amenable locally compact group. 

\begin{enumerate}[label=(\roman*)]
\item If $VN(G)$ has the w*CBAP then $G$ is weakly amenable.
\item $VN(G)$ has the w*OAP if and only if $G$ has the approximation property.
\end{enumerate}
\end{cor}

\begin{remark} From the homological perspective, the (potential) distinction between strong inner amenability and inner amenability of $\h{\G}$, is that the former case generates an $\LOQ$-module left inverse to $\Gam$ that can be approximated by normal $\LOQ$-module maps. It is not clear if a similar approximation can be achieved in the latter case, whence the strong inner amenability assumption in Proposition \ref{p:AP}.  \end{remark}

\begin{remark} To the author's knowledge, the converse of Proposition \ref{p:AP} $(i)$ is not known to hold even in the co-commutative setting. That is, if $G$ is a weakly amenable locally compact group, does $VN(G)$ have the weak*CBAP? By \cite[Theorem 3.2]{HK} it is known that in this case $\id_{VN(G)}$ can be approximated in the point-weak* topology by a bounded net in $\mc{CB}^\sigma(VN(G))$, each of whose elements is the limit in the point-weak* topology of bounded net of finite-rank elements in $\mc{CB}^\sigma(VN(G))$. \end{remark} 

\begin{prop} Let $\G$ be a topologically inner amenable locally compact quantum group.

\begin{enumerate}[label=(\roman*)]
\item If $C_0(\G)$ has the CBAP, then there exists a net $(\nu_i)$ in $M(\G)$ such that $\norm{\nu_i}_{cb}\leq\Lm_{cb}(C_0(\G))$ and $\nu_i\rightarrow 1$ $\sigma(\McbQr,\QcbQr)$.
\item If $C_0(\G)$ has the OAP then $M(\G)$ is $\sigma(\McbQr,\QcbQr)$-dense in $\McbQr$.
\end{enumerate}
\end{prop}

\begin{proof} Let $\Phi:\LOQ\rightarrow M(\G)\hten M(\G)$ be the completely positive left $\LOQ$-module map constructed in Theorem \ref{t:co-amen} from topological inner amenability. Recall that $m_{M(\G)}\circ \Phi$ is the canonical inclusion $\LOQ\hookrightarrow M(\G)$. 

$(i)$ Let $\vphi_i:C_0(\G)\rightarrow C_0(\G)$ be a net of finite-rank completely bounded maps converging to the identity in the point-norm topology such that $\norm{\vphi_i}_{cb}\leq C$. Write $\vphi_i=\sum_{k=1}^{n_i}x_k^i\ten \mu_k^i$ for some $x_k^i\in C_0(\G)$ and $\mu_k^i\in M(\G)$.  For each $i$ and $k$, $x_k^i$ is a linear combination of positive elements in $C_0(\G)$, and hence the map $(\id\ten x_k^i)\Phi\in _{\LOQ}\mc{CB}(\LOQ)$ is a linear combination of completely positive left $\LOQ$-module maps on $\LOQ$, so by \cite[Theorem 5.2]{D3} there exists $\nu_k^i$ such that $(\id\ten x_k^i)\Phi=m^r_{\nu_k^i}$.

Define $\psi_i:\LOQ\rightarrow M(\G)$ by $\psi_i=m_{M(\G)}\circ(\id\ten\vphi_i^*)\circ \Phi$.
Since $\vphi^*_i:M(\G)\rightarrow M(\G)$ forms a bounded net converging to the identity point-weak*, and $m_{M(\G)}$ is separately weak* continuous, it follows that $\psi_i$ converges to the inclusion $\LOQ\hookrightarrow M(\G)$ point-weak*. Moreover, for $f\in\LOQ$ we have
\begin{align*}\psi_i(f)&=m_{M(\G)}\circ(\id\ten\vphi_i^*)\circ \Phi(f)\\
&=\sum_{k=1}^{n_i}m_{M(\G)}((\id\ten x_k^i)\Phi(f)\ten\mu_k^i)\\
&=\sum_{k=1}^{n_i}m_{M(\G)}(f\star\nu_k^i\ten\mu_k^i)\\
&=\sum_{k=1}^{n_i}f\star\nu_k^i\star\mu_k^i\\
&=f\star\nu_i,\end{align*}
where $\nu_i=\sum_{k=1}^{n_i}\nu_k^i\star\mu_k^i\in M(\G)$ as $M(\G)$ is a closed ideal in $C_u(\G)^*$. Then $\norm{\nu_i}_{cb}\leq C$, and it follows that $\nu_i\star x\rightarrow x$ weak* in $\LIQ$ for all $x\in C_0(\G)$. Since $(\nu_i)$ is bounded in $\mc{CB}_{\LOQ}(C_0(\G),\LIQ)$, we have $\nu_i\rightarrow 1$ $\sigma(\McbQr,\QcbQr)$.

$(ii)$ The proof follows similarly to $(i)$, averaging the finite-rank maps arising from the OAP of $C_0(\G)$ and using the stable point-norm topology together with the structure of $\QcbQr$. 
\end{proof}

\section{Self-duality of Biflatness}

As in the one-sided case, for a completely contractive Banach algebra $\mc{A}$, we say that an operator $\mc{A}$-bimodule $X$ is \e{$C$-biflat} (respectively, \e{relatively $C$-biflat}) if its dual $X^*$ is $C$-injective (respectively, relatively $C$-injective) in $\AmodA$. Equivalently, $X$ is $C$-biflat if for any 1-exact sequence 
$$0\rightarrow Y\hookrightarrow Z\twoheadrightarrow Z/Y\rightarrow 0$$
in $\AmodA$, the sequence
$$0\rightarrow X_{\A}\hten_{\A}Y\hookrightarrow X_{\A}\hten_{\A}Z\twoheadrightarrow X_{\A}\hten_{\A}Z/Y\rightarrow 0$$
is $C$-exact in the category of operator spaces and completely bounded maps, where $X_{\A}\hten_{\A}Y$ is the bimodule tensor product of $X$ and $Y$, defined by $X\hten Y/_{\A}N_{\A}$, where
$$_{\A}N_{\A}=\la a\cdot x\ten y - x\ten y\cdot a, \ x'\cdot a'\ten y' - x'\ten a'\cdot y'\mid a,a'\in\A, \ x,x'\in X, \ y,y'\in Y\ra.$$

In \cite[Theorem 4.3]{RX} Ruan and Xu provided a sufficient condition for relative 1-biflatness of $\LOQH$ for any Kac algebra $\G$ by means of the existence of a certain net of unit vectors $(\xi_i)$ which are asymptotically invariant under the conjugate co-representation $W\sigma V \sigma$ and for which $\om_{\xi_i}|_{\LIQ}$ is a bounded approximate identity of $\LOQ$. In the group setting, this is precisely the quasi-SIN, or QSIN condition (see \cite{LR,St2}). We now obtain the same conclusion under a priori weaker hypotheses.

\begin{prop}\label{p:relative1-bimodule} Let $\G$ be a locally compact quantum group for which there exists a right invariant mean $m\in\LIQ^*$ satisfying
\begin{equation}\label{e:invstate}\la m,\hat{f}'\wh{\rhd}'x\ra=\la\hat{f}',1\ra\la m,x\ra, \ \ \ \hat{f}'\in\LOQHP, \ x\in\LIQ.\end{equation}
Then $\LIQ$ is relatively 1-injective in $\LOQ\hskip2pt\mathbf{mod}\hskip2pt\LOQ$. When $\G=\G_s$ is co-commutative, the converse holds.
\end{prop}

\begin{proof} It suffices to provide a completely contractive $\LOQ$-bimodule map $\Phi:\LIQ\oten\LIQ\rightarrow\LIQ$ which is a left inverse to $\Gam$. Defining $\Phi(X)=(\id\ten m)(V^*XV)$, $X\in\LIQ\oten\LIQ$, as in Proposition \ref{p:IA=>relinj}, it immediately follows that $\Phi$ is a completely contractive right $\LOQ$-module map and $\Phi\circ\Gam=\id_{\LIQ}$. However, since $m$ is also a right invariant mean on $\LIQ$, the module argument from \cite[Theorem 5.5]{CN} shows that $\Phi$ is also a left $\LOQ$-module map.

When $\G=\G_s$ is co-commutative, the converse follows from the proof of \cite[Theorem 4.1]{CT}, wherein the existence of a state $m\in VN(G)^*$ invariant under both the $A(G)$-action and the right $\LO$-action was established. Owing to the fact that $V_s=W_a$ we have
\begin{align*}f\wh{\rhd}_{s}'x&=(\id\ten f)(\h{V}'_s(x\ten 1)\h{V}_s'^*)=(f\ten\id)(V_s^*(1\ten x)V_s)\\
&=(f\ten\id)(W_a^*(1\ten x)W_a)=x\lhd_a f\end{align*}
for all $f\in\LO$ and $x\in\LG$. Hence, $m$ satisfies (\ref{e:invstate}).
\end{proof}


A completely contractive Banach algebra $\mc{A}$ is \e{operator amenable} if it is relatively $C$-biflat in $\AmodA$ for some $C\geq1$, and has a bounded approximate identity. By \cite[Proposition 2.4]{Ru3} this is equivalent to the existence of a bounded approximate diagonal in $\mc{A}\hten\mc{A}$, that is, a bounded net $(A_\alpha)$ in $\mc{A}\hten\mc{A}$ satisfying
$$a\cdot A_\alpha - A_\alpha\cdot a, \ m_{\mc{A}}(A_\alpha)\cdot a \rightarrow 0, \ \ \ a\in\mc{A}.$$
We let $OA(\mc{A})$ denote the \textit{operator amenability constant of $\mc{A}$}, the infimum of all bounds of approximate diagonals in $\mc{A}\hten\mc{A}$. This notion is the operator module analogue of the classical concept introduced by Johnson \cite{John}, who showed that the group algebra $\LO$ of a locally compact group $G$ is (operator) amenable if and only if $G$ is amenable. In \cite{Ru3}, Ruan established the dual result, showing that the Fourier algebra $A(G)$ is operator amenable precisely when $G$ is amenable. Thus, $\LO$ is operator amenable if and only if $A(G)$ is operator amenable. To motivate our next result, we now recast this equivalence at the level of (non-relative) biflatness.

\begin{prop} Let $G$ be a locally compact group. Then $\LO$ is 1-biflat if and only if $G$ is amenable if and only if $A(G)$ is 1-biflat.
\end{prop}

\begin{proof} By \cite{John} $G$ is amenable if and only if $\LO$ is amenable if and only if $\LO$ is relatively 1-biflat. Since $\LI$ is a 1-injective operator space, Proposition \ref{p:rel+inj} entails the equivalence with 1-injectivity of $\LI$ in $\LO\hskip2pt\mathbf{mod}\hskip2pt\LO$. 

Dually, if $G$ were amenable, then $\LG$ is a 1-injective operator space, and it is relatively 1-injective in $A(G)\hskip2pt\mathbf{mod}\hskip2pt A(G)$ by (the proof of) \cite[Theorem 3.6]{Ru3}. Thus, Proposition \ref{p:rel+inj} entails the 1-injectivity of $VN(G)$ as an operator $A(G)$-bimodule. Conversely, if $VN(G)$ is 1-injective in $A(G)\hskip2pt\mathbf{mod}\hskip2pt A(G)$, then it is clearly 1-injective in $A(G)\hskip2pt\mathbf{mod}$, so $G$ is amenable by \cite[Corollary 5.3]{C}.
\end{proof}

We now show that 1-biflatness of quantum convolution algebras is a self dual property.

\begin{thm}\label{t:biflat} Let $\G$ be a locally compact quantum group. Then $\LOQ$ is 1-biflat if and only if $\LOQH$ is 1-biflat.
\end{thm}

\begin{proof} Clearly, it suffices to show one direction by Pontrjagin duality, so suppose that $\LOQH$ is 1-biflat, that is, $\LIQH$ is 1-injective in $\LOQH\hskip2pt\mathbf{mod}\hskip2pt\LOQH$. Consider the canonical $\LOQH$-bimodule structure on $\BLTQ$ given by
$$\hat{f}\wh{\rhd}T=(\id\ten \hat{f})\widehat{V}(T\ten 1)\h{V}^* \hs\hs\textnormal{and}\hs\hs T\wh{\lhd}\hat{f}=(\hat{f}\ten\id)\h{W}^*(1\ten T)\h{W},$$
for $\hat{f}\in\LOQH$ and $T\in\BLTQ$. Then by 1-injectivity, $\id_{\LIQH}$ extends to a completely contractive $\LOQH$-bimodule projection $E:\BLTQ\rightarrow\LIQH$. By the left $\wh{\rhd}$-module property, it follows from the standard argument that $E(\LIQ)\subseteq\LIQ\cap\LIQH=\C1$. Also, \cite[Theorem 4.9]{CN} implies that $E$ is a right $\lhd$-module map. Let $R$ be the extended unitary antipode of $\G$. Then
$$(R\ten R)(\h{V}')=(R\ten R)(\sigma V^*\sigma)=\Sigma(R\ten R)(V^*)=\Sigma(\h{J}\ten\h{J})(V)(\h{J}\ten\h{J})=\Sigma\h{W},$$
where the last equality follows from equation (\ref{e:WV}) and the adjoint relations of $W$ and $V$. Let $E_R:\BLTQ\rightarrow\LIQHP$ be the projection of norm one, $E_R=R\circ E\circ R$. Then for $\hat{f}'\in\LOQHP$ and $T\in\BLTQ$, we have
\begin{align*}E_R(\hat{f}'\wh{\rhd}'T)&=R(E(R((\id\ten\hat{f}')\h{V}'(T\ten1)\h{V}'^*)))\\
&=R(E((\id\ten\hat{f}'\circ R)(R\ten R)(\h{V}'(T\ten1)\h{V}'^*)))\\
&=R(E((\id\ten\hat{f}'\circ R)((R\ten R)(\h{V}'^*)(R(T)\ten1)(R\ten R)(\h{V}'))))\\
&=R(E((\id\ten\hat{f}'\circ R)((\Sigma\h{W}^*)(R(T)\ten1)(\Sigma\h{W}))))\\
&=R(E((\hat{f}'\circ R\ten\id)(\h{W}^*(1\ten R(T))\h{W})))\\
&=R(E(R(T)\wh{\lhd}(\hat{f}'\circ R)))\\
&=R(E(R(T))\wh{\lhd}(\hat{f}'\circ R))\\
&=\hat{f}'\wh{\rhd}'E_R(T).
\end{align*}
Thus, $E_R$ is a left $\wh{\rhd}'$-module map. Since $R(\LIQ)=\LIQ$, the restriction $E_R|_{\LIQ}$ defines a state $m\in\LIQ^*$ satisfying
$$\la m,\h{f}'\wh{\rhd}'x\ra=\la\hat{f}',1\ra\la m,x\ra, \ \ \ \hat{f}'\in\LOQHP, \ x\in\LIQ.$$
But $E$ was also a right $\lhd$-module map, which implies that $E_R$ is a left $\rhd$-module map by the generalized antipode relation (\ref{e:R}). Thus, we also have
$$\la m,f\rhd x\ra=\la f,1\ra\la m,x\ra, \ \ \ f\in\LOQ, \ x\in\LIQ,$$
meaning that $m$ is a right invariant mean on $\LIQ$. By Proposition \ref{p:relative1-bimodule} it follows that $\LIQ$ is relatively 1-injective in $\LOQ\hskip2pt\mathbf{mod}\hskip2pt\LOQ$.

By 1-injectivity of $\LIQH$ in $\LOQH\hskip2pt\mathbf{mod}\hskip2pt\LOQH$, there exists a completely contractive morphism $\Phi:\LIQH\oten\LIQH\rightarrow\LIQH$ which is a left inverse to $\h{\Gam}$. It follows that $\Phi|_{\LIQH\ten 1}$ defines a state $\h{m}\in\LIQH^*$ which is a right $\LOQH$-module map, i.e., $\h{\G}$ is amenable. Hence, $\LIQ$ is a 1-injective operator space by \cite[Theorem 3.3]{BT}. Proposition \ref{p:rel+inj} then implies the 1-injectivity of $\LIQ$ in $\LOQ\hskip2pt\mathbf{mod}\hskip2pt\LOQ$.
\end{proof}


\begin{cor} $\ell^1(\G)$ is not relatively 1-biflat for any non-unimodular discrete quantum group.\end{cor}

\begin{proof} Since $\ell^\infty(\G)$ is always a 1-injective operator space for any disrete quantum group $\G$, if $\ell^1(\G)$ were relatively 1-biflat, then by Proposition \ref{p:rel+inj}, $\ell^\infty(\G)$ would be 1-injective in $\ell^1(\G)\hskip2pt\mathbf{mod}\hskip2pt\ell^1(\G)$. Then by Theorem \ref{t:biflat} $\LIQH$ would be 1-injective in $\LOQH\hskip2pt\mathbf{mod}\hskip2pt\LOQH$. But then, \cite[Theorem 1.1]{CLR} would entail that $\h{\G}$ is a compact Kac algebra, and therefore $\G$ is unimodular.\end{proof}

The relative biprojectivity of $\LOQ$, that is, relative projectivity of $\LOQ$ as an operator bimodule over itself, has been completely characterized: $\LOQ$ is relatively $C$-biprojective if and only if $\LOQ$ is relatively  1-biprojective if and only if $\G$ is a compact Kac algebra \cite{Ar,D3,CLR}. The corresponding characterization for (relative) $C$-biflatness remains an interesting open question. In the co-commutative setting, the relative 1-biflatness of $A(G)$ has been studied in \cite{ARS,CT,RX}. It is known to be equivalent to the existence of a contractive approximate indicator for the diagonal subgroup $G_\Delta$ \cite[Theorem 4.1]{CT}. The authors in \cite{CT} conjecture that it is equivalent to the QSIN property of $G$.

We finish this subsection with a generalization of \cite[Theorem 4.9]{KN} beyond co-amenable quantum groups, which at the same time characterizes the (non-relative) 1-biprojectivity of $\LOQ$.

\begin{thm}\label{t:birpojectivity} Let $\G$ be a locally compact quantum group. Then the following conditions are equivalent:
\begin{enumerate}[label=(\roman*)]
\item $\G$ is finite--dimensional.
\item $\mc{T}_\rhd$ is relatively 1-biprojective;
\item $\LOQ$ is 1-biprojective;
\end{enumerate}
\end{thm}

\begin{proof} $(i)\Rightarrow(ii)$ follows from \cite[Theorem 4.9]{KN}.

$(ii)\Rightarrow(iii)$ follows similarly to the proof of \cite[Theorem 5.14]{C}, giving the relative 1-biprojectivity of $\LOQ$ together with the 1-projectivity of $\LOQ$ as an operator space. The bimodule analogue of \cite[Proposition 2.2]{C} then yields $(iii)$.

$(iii)\Rightarrow(i)$ The 1-biprojectivity of $\LOQ$ ensures the existence of a normal completely bounded $\LOQ$-bimodule left inverse $\Phi:\LIQ\oten\LIQ\rightarrow\LIQ$ to $\Gam$. As usual, the restriction $\Phi|_{\LIQ\ten1}:\LIQ\rightarrow\LIQ$ maps into $\C$, and, moreover, it is a right $\LOQ$-module map, so $\G$ is compact by normality of $\Phi$. Since compact quantum groups are regular, we may repeat the proof of $(iii)\Rightarrow(i)$ from \cite[Theorem 5.14]{C} to deduce the discreteness of $\G$. Thus, $\G$ is finite-dimensional by \cite[Theorem 4.8]{KN}.
\end{proof}

\section{Operator Amenability of $L^1_{cb}(\G)$}

For a locally compact quantum group $\G$, let $L^1_{cb}(\G)$ denote the closure of $\LOQ$ inside $\McbQl$. Recall that $\h{\G}$ is weakly amenable precisely when $L^1_{cb}(\G)$ has a bounded approximate identity. In analogy to Ruan's result -- equating amenability of a locally compact group $G$ to operator amenability of $A(G)$ -- it was suggested in \cite{FRS} that $A_{cb}(G)$
may be operator amenable exactly when $G$ is weakly amenable. In \cite{CT} the authors gave examples of weakly amenable connected groups (e.g. $G=SL(2,\R)$) for which $A_{cb}(G)$ is not operator amenable. We now relate weak amenability of $\h{\G}$ to operator amenability of $L^1_{cb}(\G)$ for unimodular discrete quantum groups with Kirchberg's factorization property in the sense of \cite{BW}. 

Let $\G$ be a compact Kac algebra and let $\vphi$ and $R$ denote the Haar trace and unitary antipode on $C(\G)$, as well as their universal extensions to $C_u(\G)$. As in \cite{BW}, we define $*$-homomorphisms $\lm,\rho:C_u(\G)\rightarrow\BLTQ$ by
$$\lm(x)\Lphi(y)=\Lphi(xy), \ \ \ \rho(x)\Lphi(y)=\Lphi(y R(x)), \ \ \ x,y\in C_u(\G).$$
Since $\lm$ and $\rho$ have commuting ranges, we obtain a canonical representation $\lm\times\rho:C_u(\G)\ten_{\max}C_u(\G)\rightarrow\BLTQ$. The unimodular discrete dual $\h{\G}$ is said to have \textit{Kirchberg's factorization property} if $\lm\times\rho$ factors through $C_u(\G)\ten_{\min}C_u(\G)$. When $\G=\G_s$ is co-commutative, this notion coincides with Kirchberg's factorization property for the underlying discrete group $G$.



\begin{lem}\label{l:approx} Let $\A$ be a $C^*$-algebra. There exists a complete isometry $\iota:\A^*\hten\A^*\hookrightarrow(\A\ten_{\min}\A)^*$ such that $\iota(\A^*\hten\A^*_{\norm{\cdot}\leq1})$ is weak* dense in $(\A\ten_{\min}\A)^*_{\norm{\cdot}\leq1}$.\end{lem}

\begin{proof} Let $\pi_\A:\A\rightarrow \A^{**}$ denote the universal representation of $\A$. Then the universal cover of the representation $\pi_\A\ten\pi_\A:\A\ten_{\min}\A\rightarrow\A^{**}\ten_{\min}\A^{**}$ is a normal surjective $*$-homomorphism $\pi$ of $(\A\ten_{\min}\A)^{**}$ onto 
$$(\pi_\A\ten\pi_\A)(\A\ten_{\min}\A)''=\pi_\A(\A)''\oten\pi_\A(\A)''=\A^{**}\oten\A^{**}=(\A^*\hten\A^*)^*$$
(see \cite[Proposition IV.4.13]{T2}). Its pre-adjoint $\iota:=\pi_*:\A^*\hten\A^*\hookrightarrow(\A\ten_{\min}\A)^*$ is a complete isometry. 

Now, Let $F\in(\A\ten_{\min}\A)^*$, $\norm{F}\leq1$. Since $\pi_\A\ten\pi_\A:\A\ten_{\min}\A\rightarrow\A^{**}\ten_{\min}\A^{**}\subseteq \A^{**}\oten\A^{**}$ 
is a complete isometry, we may take a norm preserving Hahn--Banach extension $\tilde{F}\in(\A^{**}\oten\A^{**})^*=(\A^*\hten\A^*)^{**}$ satisfying
$$\la\tilde{F},\pi_\A\ten\pi_\A(A)\ra=\la F,A\ra, \ \ \ A\in\A\ten_{\min}\A.$$
By Goldstine's theorem, there exists a net $(f_j)$ in $(\A^*\hten\A^*)_{\norm{\cdot}\leq 1}$ such that $f_j\rightarrow\tilde{F}$ weak*. Thus, for all $A\in\A\ten_{\min}\A$,
$$\la F,A\ra=\la\tilde{F},\pi_\A\ten\pi_\A(A)\ra=\lim_j\la f_j,\pi_\A\ten\pi_\A(A)\ra=\lim_j\la \iota(f_j),A\ra.$$
\end{proof}

\begin{thm}\label{t:OA} Let $\h{\G}$ be a unimodular discrete quantum group with Kirchberg's factorization property. Then $\h{\G}$ is weakly amenable if and only if $L^1_{cb}(\G)$ is operator amenable. Moreover, $\Lambda_{cb}(\h{\G})\leq OA(L^1_{cb}(\G))\leq \Lambda_{cb}(\h{\G})^2$.
\end{thm}

\begin{proof} It is clear that $\Lambda_{cb}(\h{\G})\leq OA(L^1_{cb}(\G))$, as any operator amenable Banach algebra admits a bounded approximate identity with pertinent control over the norm \cite[Proposition 2.3]{Ru3}.

Conversely, by Kirchberg's factorization property the representation $\lm\times\rho$ factors through $C_u(\G)\ten_{\min}C_u(\G)$. Composing with $\om_{\Lphi(1)}$, we obtain a state $\mu:=\om_{\Lphi(1)}\circ\lm\times\rho\in(C_u(\G)\ten_{\min}C_u(\G))^*$. By $R$-invariance of $\vphi$, one can easily verify that $\mu=\mu\circ\Sigma$. Let $\Gam_u:C_u(\G)\rightarrow C_u(\G)\ten_{\min} C_u(\G)$ denote the universal co-multiplication. Then, similar to the calculations in \cite[Theorem 3.3]{Ru4}, for all $u\in\Irr$, $1\leq i,j\leq n_u$,
\begin{align*}\la\Gam_u^*(\mu)\star f,u_{ij}\ra&=\sum_{k,l=1}^{n_u}\la f,u_{lj}\ra\la\mu,u_{ik}\ten u_{kl}\ra=\sum_{k,l=1}^{n_u}\la f,u_{lj}\ra\vphi(u_{ik} R(u_{kl}))\\
&=\sum_{k,l=1}^{n_u}\la f,u_{lj}\ra\vphi(u_{ik} u_{lk}^{*})=\frac{1}{n_u}\sum_{k=1}^{n_u}\la f,u_{ij}\ra\\
&=\la f,u_{ij}\ra.\end{align*}
Hence, $\Gam_u^*(\mu)\star f = f$ for all $f\in \LOQ$. By \cite[Proposition 9.5]{K}, $R((\id\ten \vphi)(\Gam_u(x^*)(1\ten y)))=(\id\ten\vphi)((1\ten x^*)\Gam_u(y))$ for all $x,y\in C_u(\G)$, so that
\begin{align*}\la (f\ten 1)\star\mu,x\ten y\ra&=\vphi((x\star f) R(y))=(f\ten \vphi)(\Gam_u(x)(1\ten R(y)))\\
&=f\circ R(R((\id\ten\vphi)(\Gam_u(x)(1\ten R(y)))))\\
&=f\circ R((\id\ten\vphi)((1\ten x)\Gam_u(R(y))))\\
&=(f\circ R\ten\vphi)((1\ten x)\Sigma\circ R\ten R\circ\Gam_u(y))\\
&=(\vphi\ten f)(\Gam_u(y)(R(x)\ten 1))\\
&=\vphi((f\star y)R(x))\\
&=\la\mu\star(1\ten f),x\ten y\ra.\end{align*}
Thus, $(f\ten 1)\star\mu=\mu\star(1\ten f)$ for all $f\in\LOQ$. Since $\mu=\mu\circ\Sigma$, we also have $(1\ten f)\star\mu=\mu\star(f\ten 1)$, for all $f\in\LOQ$. It follows that $F\star\mu=\mu\star\Sigma F$ for all $F\in\LOQ\hten\LOQ$, where we let $\Sigma$ also denote the flip homomorphism on $\LOQ\hten\LOQ$.

The proof of Lemma \ref{l:approx} implies the existence of a net $(\mu_i)$ of states in $C_u(\G)^*\hten C_u(\G)^*$ such that $\mu_i\rightarrow\mu$ weak* in $(C_u(\G)\ten_{\min}C_u(\G))^*$, and hence in $(C_u(\G)\ten_{\max}C_u(\G))^*=C_u(\G\times\G)^*$. Since $m_{C_u(\G)^*}=\Gam_u^*|_{C_u(\G)^*\hten C_u(\G)^*}$, it follows that $m_{C_u(\G)^*}(\mu_i)\rightarrow \Gam_u^*(\mu)$ weak* in $C_u(\G)^*$. By 
\cite[Theorem 4.6]{RV} (note that $m_{C_u(\G)^*}(\mu_i),\Gam_u^*(\mu)$ are states), we have 
$$\norm{m_{C_u(\G)^*}(\mu_i)\star f-f}_{\LOQ}=\norm{m_{C_u(\G)^*}(\mu_i)\star f-\Gam_u^*(\mu)\star f}_{\LOQ}\rightarrow0, \ \ \ f\in\LOQ.$$
Since $\mu_i\rightarrow\mu$ weak* in $C_u(\G\times\G)^*$, again by \cite[Theorem 4.6]{RV} we obtain
$$\norm{F\star\mu_i-\mu_i\star\Sigma F}_{\LOQ\hten\LOQ}\rightarrow 0, \ \ \ F\in\LOQ\hten\LOQ,$$
from which we have
\begin{align*}&\norm{F\star((1\ten f)\star\mu_i-\mu_i\star(f\ten 1))}=
\norm{(F\star(1\ten f))\star\mu_i-F\star\mu_i\star(f\ten 1)}\\
&\leq\norm{(F\star(1\ten f))\star\mu_i-(\mu_i\star\Sigma F)\star(f\ten 1)}+\norm{(\mu_i\star\Sigma F)\star(f\ten 1)-F\star\mu_i\star(f\ten 1)}\\
&\leq\norm{(F\star(1\ten f))\star\mu_i-\mu_i\star(\Sigma (F\star(1\ten f)))}+\norm{\mu_i\star\Sigma F-F\star\mu_i}\norm{f\ten 1}\\
&\rightarrow0\end{align*}
for every $f\in\LOQ$, $F\in\LOQ\hten\LOQ$. Similarly, $\norm{((1\ten f)\star\mu_i-\mu_i\star(f\ten 1))\star F}\rightarrow 0$.

Now, if $\h{\G}$ is weakly amenable, then by \cite[Theorem 5.15]{KR} there exists an approximate identity $(f_j)$ for $\LOQ$ in $\mc{Z}(\LOQ)$ such that $\sup_j\norm{f_j}_{cb}<\infty$. The tensor square of the canonical complete contraction $C_u(\G)^*\rightarrow M_{cb}^l(\LOQ)$ allows us to view $\mu_i\in M_{cb}^l(\LOQ)\hten  M_{cb}^l(\LOQ)$ with $\norm{\mu_i}_{ M_{cb}^l(\LOQ)\hten M_{cb}^l(\LOQ)}\leq 1$ for all $i$. By the universal property of the operator space projective tensor product, we may also view each $\mu_i$ as an element of $\mc{CB}(M_{cb}^l(\LOQ)\hten L^1_{cb}(\G))$ by right multiplication, as well as in $\mc{CB}(L^1_{cb}(\G)\hten M_{cb}^l(\LOQ))$ by left multiplication. Moreover,
$$\norm{\mu_i}_{\mc{CB}(M_{cb}^l(\LOQ)\hten L^1_{cb}(\G))},\norm{\mu_i}_{\mc{CB}(L^1_{cb}(\G)\hten M_{cb}^l(\LOQ))}\leq\norm{\mu_i}_{ M_{cb}^l(\LOQ)\hten M_{cb}^l(\LOQ)}\leq 1.$$
Define $\mu_{ij}:=(1\ten f_j)\star\mu_i\star(f_j\ten 1)$. Then $\mu_{ij}\in L^1_{cb}(\G)\hten L^1_{cb}(\G)$ with 
$$\norm{\mu_{ij}}_{L^1_{cb}(\G)\hten L^1_{cb}(\G)}\leq\norm{f_j}_{cb}^2\leq\Lambda_{cb}(\h{\G})^2.$$
Given $f\in\LOQ$, for each $j$ we have
\begin{align*}\lim_i f\star \mu_{ij}-\mu_{ij}\star f
&=\lim_i(f\ten f_j)\star \mu_i\star(f_j\ten 1)-(1\ten f_j)\star \mu_i\star(f_j\ten f)\\
&=\lim_i(f\ten f_j^2)\star \mu_i-\mu_i\star(f_j^2\ten f)\\
&=0\end{align*}
in $\LOQ\hten\LOQ$ and therefore in $L^1_{cb}(\G)\hten L^1_{cb}(\G)$. Furthermore, 
\begin{align*}\lim_j \lim_i m_{L^1_{cb}(\G)}(\mu_{ij})\star f&=\lim_j \lim_i m_{L^1_{cb}(\G)}((1\ten f_j)\star \mu_i\star(f_j\ten 1))\star f\\
&=\lim_j \lim_i m_{\LOQ}((1\ten f_j)\star \mu_i\star(f_j\ten f))\\
&=\lim_j \lim_i m_{\LOQ}(\mu_i\star(f^2_j\ten f))\\
&=\lim_j \lim_i m_{C_u(\G)^*}(\mu_i)m_{\LOQ}(f^2_j\ten f)\\
&=\lim_j m_{\LOQ}(f^2_j\ten f)\\
&=f,\end{align*}
where the $4^{th}$ equality follows from the fact that $f_j\in\mc{Z}(\LOQ)$. Combining the iterated limit into a single net as in Proposition \ref{p:bicrossed}, we obtain a bounded approximate diagonal $(\mu_I)$ in $L^1_{cb}(\G)\hten L^1_{cb}(\G)$ with $\norm{\mu_I}_{L^1_{cb}(\G)\hten L^1_{cb}(\G)}\leq\Lambda_{cb}(\h{\G})^2$. 
\end{proof}


\begin{remark} There is a corresponding statement for the closure of $\LOQ$ in $M_{cb}^r(\LOQ)$. It is proved in the exact same way using the fact that $f\star\Gam_u^*(\mu)=f$ for all $f\in\LOQ$, which is easily verified.\end{remark}

\begin{examples}${}$\vskip5pt
\begin{enumerate}[label=(\roman*)]
\item It was shown in \cite[Theorem 2.7]{FRS} that $A_{cb}(G)$ is operator amenable for any weakly amenable discrete group $G$ such that $C^*(G)$ is residually finite-dimensional. There are examples of weakly amenable residually finite groups (e.g. $G=SL(2,\Z[\sqrt{2}])$) for which $C^*(G)$ is not residually finite-dimensional \cite{Bekka}. Since residually finite groups have Kirchberg's factorization property, Theorem \ref{t:OA} is new even for this class of discrete groups. 
\item When $\h{\G}$ is an amenable unimodular discrete quantum group, we recover Ruan's result on the operator amenability of $L^1(\G)=L^1_{cb}(\G)$ \cite[Theorem 3.5]{Ru4}. 
\item Using results from \cite{BCV} and \cite{BD}, it was shown in \cite{BW} that the discrete duals $\widehat{O_N^+}$ and $\widehat{U_N^+}$ of the free orthogonal and unitary quantum groups have Kirchberg's factorization property for $N\neq 3$. Since $\widehat{O_N^+}$ and $\widehat{U_N^+}$ are always weakly amenable with Cowling--Haagerup constant 1 \cite{F}, we have $OA(L^1_{cb}(O_N^+))=OA(L^1_{cb}(U_N^+))=1$ for all $N\neq 3$. 
\item If $\G_1$ and $\G_2$ are compact quantum groups with Kirchberg's factorization property and $\Lambda_{cb}(\h{\G_1})=\Lambda_{cb}(\h{\G_2})=1$, then $\G=\G_1\ast\G_2$ also has the factorization property \cite{BD} and $\Lambda_{cb}(\h{\G})=1$ \cite{F2}, so $OA(L^1_{cb}(\G))=1$.
\end{enumerate}
\end{examples}

\begin{remark} It would be interesting to find an example of a unimodular discrete quantum group $\h{\G}$ with Kirchberg's factorization property for which equality of the constants in Theorem \ref{t:OA} does not hold.\end{remark} 

\section{Decomposability}

For a locally compact group $G$, it is well-known that $B(G)=\Mcb$ completely isometrically whenever $G$ is amenable \cite[Corollary 1.8]{dH}. We now generalize this implication to arbitrary locally compact quantum groups. Moreover, we show that the corresponding complete isometry is a weak*-weak* homeomorphism.

By \cite[\S6]{K}, the universal co-representation $\W_{\G}=(\id\ten\pi_{\h{\G}})(\mathbb{W}_{\G})\in M(C_u(\G)\ten_{\min} C_0(\h{\G}))$ from the proof of Proposition \ref{p:qs}  satisfies
\begin{enumerate}[label=(\roman*)]
\item $\hat{\lm}_u(\h{f})=(\id\ten\h{f})(\W_{\G}^*)$ for all $\h{f}\in\LOQHs$,
\item $(\id\ten\pi_{\G})\circ\Gam_u(x)=\W_{\G}^*(1\ten\pi_{\G}(x))\W_{\G}$ for all $x\in C_u(\G)$,
\end{enumerate}
where $\hat{\lm}_u$ is the embedding of $\LOQHs$ into $C_u(\G)$, and $\pi_{\G}:C_u(\G)\rightarrow C_0(\G)$ is the (unique) extension of $\hat{\lm}:\LOQHs\rightarrow C_0(\G)$ \cite[Proposition 5.1]{K}. Moreover,
\begin{equation}\label{C_u density}C_u(\G)=\overline{\text{span}\{(\id\ten\h{f})(\W_{\G}) : \h{f}\in\LOQH\}}^{\norm{\cdot}_u}.\end{equation}

We will need the following representation of $\Theta^l(\mu)$ for $\mu\in C_u(\G)^*$, which may be found in \cite[Theorem 6.1]{D}. We present the proof for the convenience of the reader.

\begin{lem}\label{l:lemma} For $\mu\in C_u(\G)^*$, 
$$\Theta^l(\mu)(x)=(\mu\ten\id)\W_{\G}^*(1\ten x)\W_{\G}, \ \ \ x\in\LIQ.$$\end{lem}

\begin{proof} First let $x=\hat{\lm}(\hat{f})\in C_0(\G)$ for some $\hat{f}\in\LOQHs$. Then, for all $f\in\LOQ$,
\begin{align*}\la\Theta^l(\mu)(x),f\ra&=\la x,m^l_\mu(f)\ra=\la\hat{\lm}(\hat{f}),\mu\star f\ra\\
                                 &=\la\pi_{\G}\circ\hat{\lm}_u(\hat{f}),\mu\star f\ra=\la\hat{\lm}_u(\hat{f}),\mu\star_u \pi_{\G}^*(f)\ra\\
                                 &=\la\Gam_u(\hat{\lm}_u(\hat{f})),\mu\ten \pi_{\G}^*(f)\ra=\la(\id\ten\pi_{\G})(\Gam_u(\hat{\lm}_u(\hat{f}))),\mu\ten f\ra\\
                                 &=\la\W_{\G}^*(1\ten\pi_{\G}(\hat{\lm}_u(\hat{f})))\W_{\G},\mu\ten f\ra=\la\W_{\G}^*(1\ten x)\W_{\G},\mu\ten f\ra\\
                                 &=\la(\mu\ten\id)\W_{\G}^*(1\ten x)\W_{\G},f\ra.\end{align*}
As $\hat{\lm}(\LOQHs)$ is norm dense in $C_0(\G)$, and since $C_0(\G)$ is weak* dense in $\LIQ$, the result follows.\end{proof}

Recall that $\tilde{\lm}:C_u(\G)^*\rightarrow\McbQl$ is the map taking $\mu\in C_u(\G)^*$ to the operator of left multiplication by $\mu$ on $\LOQ$.

\begin{thm}\label{B(G)=McbA(G)} Let $\G$ be a locally compact quantum group. If $\h{\G}$ is amenable then $\tilde{\lm}:C_u(\G)^*\rightarrow\McbQl$ is a weak*--weak* homeomorphic completely isometric isomorphism.\end{thm}

\begin{proof} Amenability of $\h{\G}$ entails the surjectivity of $\tilde{\lm}$ from (the left version of) \cite[Proposition 5.10]{C}. For simplicity, throughout the proof we denote by $\Theta^l(\mu)$ the map $\Theta^l(\tilde{\lm}(\mu))$ for $\mu\in C_u(\G)^*$.

In \cite[Theorem 5.2]{D} Daws shows that $\Theta^l:C_u(\G)^*_+\rightarrow \ _{\LOQ}\mc{CP}^\sigma(\LIQ)$ is an order bijection. We show that it is a complete order bijection. To this end, let $[\mu_{ij}]\in M_n(C_u(\G)^*)^+$. By Lemma \ref{l:lemma}
$$\Theta^l(\mu_{ij})(x)=(\id\ten\mu_{ij})\W_{\G}^*(1\ten x)\W_{\G}, \ \ \ x\in\LIQ.$$
Thus, for any $x_1,..,x_m\in\LIQ$ we have
\begin{align*}((\Theta^l)^n([\mu_{ij}]))^m([x_k^*x_l])&=[(\mu_{ij}\ten\id)\W_{\G}^*(1\ten x_k^*x_l)\W_{\G}]\\
&=[(\mu_{ij}\ten\id)]^m([\W_{\G}^*(1\ten x_k^*x_l)\W_{\G}])\geq0.\end{align*}
It follows that $(\Theta^l)^n([\mu_{ij}])\in\mc{CP}(\LIQ,M_n(\LIQ))$.

On the other hand, suppose $[\mu_{ij}]\in M_n(C_u(\G)^*)$ such that 
$$(\Theta^l)^n([\mu_{ij}])\in\mc{CP}(\LIQ,M_n(\LIQ)).$$
Let $\h{P}$ denote the positive operator implementing the scaling group $\hat{\tau}$ on $\h{\G}$, via $\hat{\tau}_t(\hat{x})=\h{P}^{it}\hat{x}\h{P}^{-it}$, $\hat{x}\in\LIQH$. Using \cite[Proposition 6.1]{D3}, for $\xi_1,...\xi_m\in\mc{D}(\h{P}^{1/2})$, $\eta_1,...,\eta_m\in\mc{D}(\h{P}^{-1/2})$, and $[z_{ik}]\in\C^{nm}$,
\begin{align*}
\la[\mu_{ij}]^m([\lm_u(\om_{\xi_k,\eta_k})^*\lm_u(\om_{\xi_l,\eta_l})])[z_{ik}],[z_{ik}]\ra
&=\sum_{i,j=1}^n\sum_{k,l=1}^m\overline{z_{ik}}z_{jl}\la \mu_{ij},\lm_u(\om_{\xi_k,\eta_k})^*\lm_u(\om_{\xi_l,\eta_l})\ra\\
&=\overline{\sum_{i,j=1}^n\sum_{k,l=1}^mz_{ik}\overline{z_{jl}}\la \mu_{ij}^*,\lm_u(\om_{\xi_l,\eta_l})\lm_u(\om_{\xi_k,\eta_k})^*\ra}\\
&=\overline{\sum_{i,j=1}^n\sum_{k,l=1}^mz_{ik}\overline{z_{jl}}\la \Theta^l(\mu_{ij})(\xi_l\xi_k^*)\eta_k,\eta_l\ra}\\
&=\overline{\sum_{k,l=1}^m \la[\Theta^l(\mu_{ij})(\xi_l\xi_k^*)]y_k,y_l\ra}\geq0,\end{align*}
where $y_k=[z_{1k}\eta_k \ \cdots \ z_{nk}\eta_k]^T\in\LTQ^n$ for $1\leq k\leq m$. By density of $\{\om_{\xi,\eta}\mid \xi\in\mc{D}(\h{P}^{1/2}), \eta\in\mc{D}(\h{P}^{-1/2})\}$ in $\LOQHs$ \cite[Lemma 3]{DS}, it follows that $[\mu_{ij}]\in M_n(C_u(\G)^*)^+$.

We now show that $\tilde{\lm}$ is a complete isometry. To do so we introduce a decomposability norm on $_{\LOQ}\mc{CB}^{\sigma}(\LIQ)$, given by
$$\norm{\Phi}_{L^1dec}:=\inf\bigg\{\max\{\norm{\Psi_1}_{cb},\norm{\Psi_2}_{cb}\}\mid \begin{bmatrix}\Psi_1 & \Phi\\ \Phi^* & \Psi_2\end{bmatrix}\geq_{cp}0\bigg\},$$
where $\Psi_i\in \ _{\LOQ}\mc{CP}^{\sigma}(\LIQ)$. It is evident that $\norm{\cdot}_{L^1dec}$ is a norm on $_{\LOQ}\mc{CB}^{\sigma}(\LIQ)$. That $\norm{\Phi}_{cb}\leq\norm{\Phi}_{L^1dec}$ for all $\Phi\in \ _{\LOQ}\mc{CB}^{\sigma}(\LIQ)$ follows verbatim from the first part of \cite[Lemma 5.4.3]{ER}. In a similar fashion we obtain a decomposable norm on 
$$M_n(_{\LOQ}\mc{CB}^{\sigma}(\LIQ)= \ _{\LOQ}\mc{CB}^{\sigma}(\LIQ,M_n(\LIQ)).$$
Since $\Theta^l$ is a completely positive contraction from $(C_u(\G)^*)^+$ onto $_{\LOQ}\mc{CP}^{\sigma}(\LIQ)$, one easily sees that
$$\norm{(\Theta^l)^n([\mu_{ij}])}_{L^1dec}\leq\norm{[\mu_{ij}]}_{dec}, \ \ \ [\mu_{ij}]\in M_n(C_u(\G)^*),$$
where $\norm{\cdot}_{dec}$ is the standard decomposable norm for maps between $C^*$-algebras. 

Conversely, if 
$\norm{(\Theta^l)^n([\mu_{ij}])}_{dec}<1$ then there exist 
$$\Psi_1,\Psi_2\in \  _{\LOQ}\mc{CP}^{\sigma}(\LIQ,M_n(\LIQ))$$
such that $\norm{\Psi_1}_{cb},\norm{\Psi_2}_{cb}<1$, and
$$\begin{bmatrix}\Psi_1 & (\Theta^l)^n([\mu_{ij}])\\ (\Theta^l)^n([\mu_{ij}])^* & \Psi_2\end{bmatrix}\geq_{cp}0.$$
Since $(\Theta^l)^n$ is a complete order bijection there exist $[\nu^k_{ij}]\in M_n(C_u(\G)^*)^+=\mc{CP}(C_u(\G),M_n)$ such that $\Psi_k=(\Theta^l)^n([\nu^k_{ij}])$, $k=1,2$, and 
$$\begin{bmatrix}[\nu^1_{ij}] & [\mu_{ij}]\\ [\mu_{ij}]^* & [\nu^2_{ij}]\end{bmatrix}\geq_{cp}0.$$
It follows that $[\nu^k_{ij}]$ is a strictly continuous completely positive map $C_u(\G)\rightarrow M_n$, and therefore admits a unique extension to a completely positive map $\widetilde{[\nu^k_{ij}]}:M(C_u(\G))\rightarrow M_n$ which is strictly continuous on the unit ball \cite[Corollary 5.7]{Lance}. By uniqueness, $\widetilde{[\nu^k_{ij}]}=[\tilde{\nu}^k_{ij}]$, where $\tilde{\nu}^k_{ij}$ is the unique strict extension of the functional $\nu^k_{ij}$. Thus, by completely positivity
$$\norm{[\nu^k_{ij}] }_{cb}=\norm{\widetilde{[\nu^k_{ij}]}(1_{M(C_u(\G))})}=\norm{[\tilde{\nu}^k_{ij}(1_{M(C_u(\G))})]} =\norm{\Psi_k(1)}=\norm{\Psi_k}_{cb}<1,$$
so that $\norm{[\mu_{ij}]}_{dec}<1$. Therefore
$$\norm{(\Theta^l)^n([\mu_{ij}])}_{L^1dec}=\norm{[\mu_{ij}]}_{dec}.$$
However, $\norm{[\mu_{ij}]}_{dec}=\norm{[\mu_{ij}]}_{cb}$ by injectivity of $M_n$ (see \cite[Lemma 5.4.3]{ER}), so that
$$\norm{(\Theta^l)^n([\mu_{ij}])}_{L^1dec}=\norm{[\mu_{ij}]}_{cb}, ,\ \ \ [\mu_{ij}]\in M_n(C_u(\G)^*).$$


Now, amenability of $\h{\G}$ entails the the 1-injectivity of $\LIQ$ in $\LOQ\hskip2pt\mathbf{mod}$ by the left version of \cite[Theorem 5.1]{C}. The matricial analogues of the proofs of \cite[Proposition 5.5, Lemma 5.7]{C} show that 
$$\norm{(\Theta^l)^n([\mu_{ij}])}_{cb}=\norm{(\Theta^l)^n([\mu_{ij}])}_{L^1dec}=\norm{[\mu_{ij}]}_{cb}.$$
Hence, $\Theta^l:C_u(\G)^*\rightarrow \ _{\LOQ}\mc{CB}^\sigma(\LIQ)$ is a completely isometric isomorphism. To prove that $\Theta^l$ is a weak*-weak* homeomorphism, it suffices to show that it is weak* continuous on bounded sets. Let $(\mu_i)$ be a bounded net in $C_u(\G)^*$ converging weak* to $\mu$. Since $C_u(\G)^*$ is a dual Banach algebra \cite[Lemma 8.2]{D4}, multiplication is separately weak* continuous. Hence, for $\hat{f}\in\LOQH$ and $f\in\LOQ$,
$$\la\Theta^l(\mu_i)(\hat{\lm}(\hat{f})),f\ra=\la\hat{\lm}_u(\hat{f}),\mu_i\star_u \pi_{\G}^*(f)\ra\rightarrow\la\hat{\lm}_u(\hat{f}),\mu\star_u \pi_{\G}^*(f)\ra
=\la\Theta^l(\mu_i)(\hat{\lm}(\hat{f})),f\ra.$$
The density of $\hat{\lm}(\LOQH)$ in $C_0(\G)$ and boundedness of $\Theta^l(\mu_i)$ in $_{\LOQ}\mc{CB}(C_0(\G),\LIQ)=(Q_{cb}^l(
\G))^*$ (see \cite[Proposition 5.8]{C}) establish the claim.
\end{proof}

\begin{remark} We note that the conclusion of Theorem \ref{B(G)=McbA(G)} was obtained under the a priori stronger assumption that $\G$ is co-amenable \cite[Theorem 4.2]{HNR2}.\end{remark}

\begin{cor} Let $\G$ be a locally compact quantum group such that $\h{\G}$ has the approximation property. Then $\G$ is co-amenable if and only if $\h{\G}$ is amenable.\end{cor}

\begin{proof} Assuming $\h{\G}$ has the AP, there exists a stable approximate identity $(f_i)$ for $\LOQ$. Moreover, as noted in the proof of Proposition \ref{p:AP}, $\LIQ$ has the w*OAP, and therefore the dual slice map property (see \cite[Theorem 11.2.5]{ER}). Any operator space is a complete quotient of the space of trace class operators for some Hilbert space \cite[Corollary 3.2]{Blech}, so let $H$ be a Hilbert space such that $\Th\twoheadrightarrow C_0(\G)$. Then $\LOQ\hten\Th\twoheadrightarrow\LOQ\hten C_0(\G)$ by projectivity of $\hten$, and $\LIQ\oten M(\G)\hookrightarrow\LIQ\oten\BH$ is a weak*-weak* continuous complete isometry. Hence,
$$\Theta^l(f_i)\ten\id_{M(\G)}(X)\rightarrow X,$$
weak* for all $X\in \LIQ\oten M(\G)$, so that
$$\Theta^l(f_i)_*\ten\id_{C_0(\G)}(A)\rightarrow A$$
weakly for all $A\in\LOQ\hten C_0(\G)$. By the standard convexity argument, we may assume that the net $(f_i)$ satisfies
$$\norm{\Theta^l(f_i)_*\ten\id_{C_0(\G)}(A) - A}\rightarrow0, \ \ \ A\in\LOQ\hten C_0(\G).$$
Consider the multiplication map $m:\LOQ\hten C_0(\G)\rightarrow C_0(\G)$. Let $\tilde{m}:\LOQ\hten_{\LOQ}C_0(\G)\rightarrow C_0(\G)$ denote the induced map and $q:\LOQ\hten C_0(\G)\twoheadrightarrow \LOQ\hten_{\LOQ}C_0(\G)$ denote the quotient map. It follows that $q(f\star A)=q(f\ten m(A))$ for all $f\in\LOQ$ and $A\in\LOQ\hten C_0(\G)$. Thus, if $m(A)=0$, then
$$q(A)=\lim_i q(f_i\star A)=\lim_iq(f_i\ten m(A))=0,$$
so that the induced multiplication $\tilde{m}$ is injective.

Now, assuming $\h{\G}$ is amenable, Theorem \ref{B(G)=McbA(G)} implies that $C_u(\G)^*\cong \  _{\LOQ}\mc{CB}(C_0(\G),\LIQ)$ completely isometrically and weak*-weak* homeomorphically, that is, $C_u(\G)\cong\LOQ\hten_{\LOQ} C_0(\G)$ completely isometrically. By the left version of \cite[Theorem 5.1]{C} $\LOQ$ is 1-flat in $\mathbf{mod}\hskip2pt\LOQ$. Thus, the following sequence is 1-exact
$$0\rightarrow\LOQ\hten_{\LOQ}\mathrm{Ker}(\pi_{\G})\hookrightarrow\LOQ\hten_{\LOQ}C_u(\G)\twoheadrightarrow\LOQ\hten_{\LOQ}C_0(\G)\rightarrow0.$$
Since $f\star x=(\id\ten f\circ\pi_{\G})\circ\Gam_u(x)=(\id\ten f)\W_{\G}^*(1\ten\pi_{\G}(x))\W_{\G}=0$ for all $x\in\mathrm{Ker}(\pi_{\G})$, it follows that $\LOQ\hten_{\LOQ}\mathrm{Ker}(\pi_{\G})=0$. Hence, $\LOQ\hten_{\LOQ}C_u(\G)\cong\LOQ\hten_{\LOQ}C_0(\G)$. Moreover, as
$$\Theta^l(\mu\star f)(x)=\Theta^l(f)(\Theta^l(\mu)(x))=(\Theta^l(\mu)(x))\star f, \ \ \ f\in\LOQ, \ \mu\in C_u(\G)^*, \ x\in C_0(\G),$$
it follows that $C_u(\G)\cong\LOQ\hten_{\LOQ} C_0(\G)$ is an isomorphism of left $\LOQ$-modules, i.e., $C_u(\G)$ is induced. The commutative diagram
\begin{equation*}
\begin{tikzcd}
\LOQ\hten_{\LOQ}C_u(\G) \arrow[r, equal]\arrow[d, equal] &\LOQ\hten_{\LOQ}C_0(\G)\arrow[d, "\tilde{m}"]\\
C_u(\G) \arrow[r, two heads] &C_0(\G)
\end{tikzcd}
\end{equation*}
then implies that $\tilde{m}$ is a complete quotient map. Thus, $C_0(\G)$ is an induced $\LOQ$-module and 
$$M(\G)\cong (\LOQ\hten_{\LOQ}C_0(\G))^*= \ _{\LOQ}\mc{CB}(C_0(\G),\LIQ).$$
The measure corresponding to the inclusion $C_0(\G)\hookrightarrow\LIQ$ is necessarily a left unit for $M(\G)$, which entails the co-amenability of $\G$.

\end{proof}

\section*{Acknowledgements}

The author would like to thank Michael Brannan and Ami Viselter for helpful discussions at various points during this project, as well as the anonymous referee whose valuable comments significantly improved the presentation of the paper. The author was partially supported by the NSERC Discovery Grant 1304873.

\end{spacing}

\end{document}